\documentclass[12pt]{amsart}
\usepackage{amsmath,amssymb,amsbsy,amsfonts,amsthm,latexsym,
                        amsopn,amstext,amsxtra,euscript,amscd,mathrsfs,color,bm}
\usepackage[utf8]{inputenc}

\usepackage[x11names]{xcolor}
\usepackage{ae,aecompl}

\usepackage[colorlinks,linkcolor=blue,anchorcolor=blue,citecolor=blue,backref=page]{hyperref}

\newtheorem{theorem}{Theorem}
\newtheorem{lemma}[theorem]{Lemma}

\newtheorem{proposition}[theorem]{Proposition}

\theoremstyle{definition}

\newtheorem{remark}[theorem]{Remark}

\numberwithin{equation}{section}
\numberwithin{theorem}{section}
\numberwithin{table}{section}
\numberwithin{figure}{section}

\newcommand{\C}{\mathbb{C}}

\newcommand{\Q}{\mathbb{Q}}

\newcommand{\Z}{\mathbb{Z}}

\newcommand{\OQq}{\overline{\mathbb{Q}}}

\newcommand{\cP}{\mathcal{P}}

\newcommand{\al}{\alpha}

\newcommand{\ga}{\gamma}

\newcommand{\de}{\delta}

\newcommand{\fp}{\frak{p}}

\newcommand{\fd}{\frak{d}}

\newcommand{\fA}{\frak{A}}
\newcommand{\fB}{\frak{B}}
\newcommand{\fC}{\frak{C}}

\newcommand{\fP}{\frak{P}}

\newcommand{\h}{\mathrm{h}}

\newcommand{\ord}{\mathrm{ord}}
\newcommand{\hh}{\hat{\h}}
\newcommand{\Tr}{\mathrm{Tr}}

\begin{document}

\title[Superelliptic equations over number fields]{Explicit bounds for the solutions of superelliptic equations over number fields}

\author[A. B\' erczes]{Attila B\' erczes}
\address{Institute of Mathematics, University of Debrecen, H-4010 Debrecen, P.O. BOX 12, Hungary}
\email{berczesa@science.unideb.hu}

\author[Y. Bugeaud]{Yann Bugeaud}
\address{Institut de Recherche Math\' ematique Avanc\' ee, U.M.R. 7501, Universit\' e de Strasbourg et C.N.R.S., 7, rue Ren\' e Descartes, 67084 Strasbourg, France}

\address{Institut universitaire de France}
\email{yann.bugeaud@math.unistra.fr}

\author[K. Gy\H ory]{K\'alm\'an Gy\H ory}
\address{Institute of Mathematics, University of Debrecen, H-4010 Debrecen, P.O. Box 12, Hungary}
\email{gyory@science.unideb.hu}

\author[J. Mello]{Jorge Mello}
\address{Department of Mathematics and Statistics, Oakland University, Michigan}
\email{jmelloguitar@gmail.com}

\author[A. Ostafe]{Alina Ostafe}
\address{School of Mathematics and Statistics, University of New South Wales, Sydney, NSW 2052, Australia}
\email{alina.ostafe@unsw.edu.au}

\author[M. Sha]{Min Sha}
\address{School of Mathematical Sciences, South China Normal University, Guangzhou 510631, China}
\email{min.sha@m.scnu.edu.cn}



\begin{abstract}
Let $f$ be a polynomial with coefficients in the ring $O_S$ of $S$-integers of a number field $K$,
$b$ a non-zero $S$-integer, and $m$ an integer $\ge 2$.
We consider the equation $( \star )$: $f(x) = b y^m$ in $x,y \in O_S$.
Under the well-known LeVeque condition, we give fully explicit upper bounds in terms of $K, S, f, m$ and the $S$-norm of $b$ for the heights of the solutions $x$ of the equation $( \star)$.
Further, we give an explicit bound $C$ in terms of $K, S, f$ and the $S$-norm of $b$ such that if $m > C$
the equation $(\star)$ has only solutions with $y = 0$ or a root of unity.
Our results are more detailed versions of work of Trelina, Brindza, Shorey and Tijdeman, Voutier and Bugeaud,
and extend earlier results of B\'erczes, Evertse, and Gy\H ory to polynomials with multiple roots.
In contrast with the previous results, our bounds depend on the $S$-norm of $b$
instead of its height. 
\end{abstract}

\maketitle

\section{Introduction}

Let $f$ be an integer polynomial of degree $n$ without multiple roots and $m$ an integer $\ge 2$.
Siegel \cite{Siegel26,Siegel29} proved that the equation
\begin{equation}  \label{eq:fxym}
f(x) = y^m  
\end{equation}
has only finitely many solutions in integers $x, y$ if $mn \ge 6$.
This assumption is sharp. Indeed, for any positive integer $d$ which is not the square of an integer, the
Pell equation $x^2 - d y^2 = 1$ has infinitely many integer solutions $x, y$.
Siegel's proof is ineffective and does not yield any upper bound for the absolute values of the solutions $x, y$.
The first explicit upper bounds were given by Baker in 1969, by means of his theory of linear forms in the
logarithms of algebraic numbers.
A few years later, in 1976, Schinzel and Tijdeman \cite{SchTij} were able to treat \eqref{eq:fxym} with the exponent $m$ unknown.
Namely, they proved that there is an effectively computable number $C$, depending only on $f$,
such that \eqref{eq:fxym} has no integer solutions $x, y$ with $y \notin \{0, \pm 1\}$ if $m > C$. Their proof is also based on Baker's
theory of linear forms in the logarithms of algebraic numbers. Up to now, we have no alternative proof of their theorem.
All these results have subsequently been generalized over number fields and over finitely
generated domains over $\Z$; see \cite{BEG,BA19} and Section 2 below for references.

Let now $f$ be a polynomial with coefficients in the ring $O_S$ of $S$-integers of a given number field $K$,
$b$ a non-zero $S$-integer, and $m$ an integer $\ge 2$.
Assume that $f$ has no multiple roots.
Fully explicit upper bounds for the heights of the solutions $x, y$ in $O_S$ to
\begin{equation} \label{eq:fxbym}
f(x) = b y^m 
\end{equation}
have been derived in \cite{BEG}. In the same paper, the authors
also gave an explicit value for $C$ such that, if $m > C$, then \eqref{eq:fxbym}
has only solutions with $y = 0$ or a root of unity.

The purpose of the present paper is to extend all the results of \cite{BEG} to polynomials with multiple roots.
We also considerably improve the dependence on $b$ in the bounds. Namely, our bounds
involve the $S$-norm of $b$ instead of its height, which may be considerably larger.

Furthermore, our Theorem 2.2 below is a crucial ingredient for proving Theorem 1.5 in a forthcoming paper of ours \cite{BA24}, which is about multiplicative dependence of rational values modulo approximate finitely generated groups; see Section~\ref{sec:application} for some more details.


\section{Main results}

\subsection{Notation}\label{sec:notation}
Throughout the paper, let $K$ be a  number field of degree $d$ with discriminant $D_{K}$, ring of integers $O_{K}$ and set of places $M_{K}$.
For any prime ideal $\fp$ of $O_{K}$, let 
$$N_K(\fp)= \# O_K / \fp$$
 be the norm of $\fp$. 
As usual, we define the norm $N_K(\mathfrak{I})$ of any fractional ideal $\mathfrak{I}$ of $K$ through its prime ideal decomposition. 

Let $S$ be a finite subset of $M_K$ containing all infinite places and $O_S$ the ring of $S$-integers in $K$.
For $\alpha \in K$, let $N_{K/\Q}(\alpha)$ be the norm of $\alpha$ from $K$ to the rational numbers $\Q$, and we define the {\it $S$-norm} of $\alpha$ by
\begin{equation}  \label{eq:S-norm}
N_S(\alpha)=\prod_{v\in S}|\alpha|_v,
\end{equation}
 where  $|\cdot|_v$ is the normalized absolute valuation of $K$ at the place $v$; 
and when $\alpha \ne 0$, we denote by $\h(\alpha)$ the absolute logarithmic Weil height of $\alpha$. 
Clearly, by definition we have (see, for instance, \cite[Equation (1.9.2)]{EGy10})
\begin{equation} \label{eq:NS}
\log N_S(\alpha) \le d \h(\alpha),
\end{equation}
where $\log$ is the natural logarithm. 

For any polynomial $f= a_0 X^n + \cdots + a_n \in K[X]$ of degree $n \ge 1$,
as usual its \textit{absolute (multiplicative) height},
denoted by $H(f)$, is defined to be
$$
H(f)=\prod_{v\in M_K} \max \left( 1,|a_0|_v,\dots , |a_n|_v\right)^{1/[K:\Q]}
$$
and its \textit{absolute logarithmic Weil height}, denoted by $\h(f)$, is defined as
$$
\h(f)=\log H(f). 
$$
We define its \textit{homogeneous height}, denoted by $\hh(f)$,  to be
the absolute logarithmic Weil height of the projective vector $[a_0, \ldots, a_n]$.
Clearly, we have
$$
\hh(f) \le \h(f).
$$
These heights do not depend on the field $K$ containing the coefficients of $f$.

 Put $s=\# S$ and
\begin{eqnarray*}
&&P_S=Q_S = 1, \ \ \mbox{if $S$ consists only of infinite places;}
\\
&&P_S = \max_{i=1, \ldots, s^\prime} N_K(\fp_i),\ \ Q_S = \prod_{i=1}^{s^\prime} N_K(\fp_i),
\\
&&\qquad\qquad\mbox{if $\fp_1,\dots,\fp_{s^\prime}$ are the prime ideals in $S$.}
\end{eqnarray*}

In addition, we  define $\log^* x = \max(1, \log x)$ for any $x > 0$.

\subsection{Bounding the solutions}

Let
$$
f(X) = a_0X^n+a_1X^{n-1}+\cdots+a_n \in O_S[X]
$$
be a polynomial of degree $n \geq 2$. Let $b$ be a non-zero element of $O_S, m\geq 2$ an integer and consider the equation
\begin{equation}\label{eq:super}
    f(x)=by^m \qquad  \text{ in } x,y \in O_S.
\end{equation}
 We assume that in some finite extension of $K$, the polynomial
$f \in O_S[X]$  factorizes as
$$
f(X) = a_0\cdot \prod_{i=1}^r (X-\alpha_i)^{e_i}
$$
with distinct $\alpha_1, \dots, \alpha_r$. Put
$$
m_i=\frac{m}{\gcd(m,e_i)}, \ \ i=1,\dots , r.
$$
We may and shall assume throughout the paper that
\begin{equation}\label{eq:m_decreasing}
    m_1\geq m_2 \geq \cdots \geq m_r.
\end{equation}
Let
$$
f^*(X) = a_0\cdot \prod_{i=1}^r (X-\alpha_i), 
$$
and let $D(f^*)$ denote the discriminant of $f^*$. It is clear that the polynomial $f^*$ has its coefficients in $O_S$.

Improving a classical result of Siegel \cite{Siegel26}, LeVeque \cite{Leveque64} showed that
if $S$ consists only of infinite places and $b=1$ then the equation \eqref{eq:super} has only finitely many solutions, unless possibly $(m_1,\dots, m_r)$ is one of the $r$-tuples
\begin{equation}\label{eq:LeVeque2}
(k,1,\dots,1), \quad\quad k\geq 1,
\end{equation}
or
\begin{equation}\label{eq:LeVeque1}
(2,2,1,\dots ,1),
\end{equation}
when \eqref{eq:super} may have infinitely many solutions. 

LeVeque's theorem is ineffective. It has been made effective
by Brindza \cite{B6}; see also Shorey and Tijdeman \cite{ShT}.
Namely, Brindza gave an effective bound for the heights of the solution
$x,y$ of \eqref{eq:super}, provided that $b=1$ and $(m_1,\dots, m_r)$ is not of the
form \eqref{eq:LeVeque2} or \eqref{eq:LeVeque1}, and equivalently is of the form 
\begin{equation}\label{eq:LeVeque3}
m_1\geq 3, m_2\geq 2 \quad\quad \text{or} \quad\quad  m_1=m_2=m_3=2, 
\end{equation}
which is the so-called LeVeque condition. 

Under certain restriction Trelina \cite{Trelina78}, Poulakis \cite{Poulakis}, Voutier \cite{Voutier2}, Bugeaud
\cite{Bug2} and B\'erczes, Evertse and Gy\H ory \cite{BEG} gave more explicit bounds
for the solutions of equation \eqref{eq:super}.
Here we give completely explicit bounds for the solutions of \eqref{eq:super} without any restriction, subject to the condition that \eqref{eq:LeVeque3} holds. We prove the following.

\begin{theorem}\label{T_LeVeque}
Suppose that $m_1\geq m_2 \geq \cdots \geq m_r$ and that $(m_1,\dots, m_r)$
is not of the form \eqref{eq:LeVeque2} or \eqref{eq:LeVeque1}.
\begin{enumerate}
\item[(i)] If $m_1=m_2=m_3=2$, then all solutions $(x,y)$ of \eqref{eq:super} satisfy
    \begin{equation}\label{eq:bound222}
    \h(x) \leq (16r^3 s)^{80r^4s}|D_K|^{8r^3}H(f^*)^{50dr^4} Q_S^{20r^3}N_S(b)^{16r^3}.
    \end{equation}

\item[(ii)] Let $m\geq 3$. If there exist $1\leq i<j\leq r$ such that $\gcd(m_i,m_j)\geq 3$, then all solutions $(x,y)$ of \eqref{eq:super} satisfy
    \begin{equation}\label{eq:bound_m'}
    \h(x) \leq (6rs)^{14m^3r^3s}|D_K|^{2m^2r^2}H(f^*)^{8dm^2r^3} Q_S^{3m^2r^2}N_S(b)^{2m^2r^2}.
    \end{equation}

\item[(iii)] Let $m\geq 3$, and suppose that neither the hypothesis in (i) nor the hypothesis in (ii) are satisfied. Then there exist $1\leq i<j\leq r$ such that
    $m_i\geq 3$ and  $m_j\geq 2$ and all solutions $(x,y)$ of \eqref{eq:super} satisfy
    \begin{equation}\label{eq:bound32}
    \h(x) \leq (2r)^{16dm^{9}r^3}(12m^5s)^{28m^{10}s}|D_K|^{2m^8r^2}H(f^*)^{32dm^9r^3}
    Q_S^{6m^8r^2}N_S(b)^{4m^8r^2}.
    \end{equation}
\end{enumerate}
\end{theorem}

We remark that the upper bounds for $\h(x)$ in Theorem~\ref{T_LeVeque} immediately give upper bounds for 
 $\h(y)$ via the equation \eqref{eq:super}, which however involve $\h(b)$. 
Besides, the dependence on $H(f^*)$ can be transferred to $H(f)$ via Lemma~\ref{lem:f*} below. 

We note that if $m\geq 3$ and $f(X)$ has only simple roots, then Theorem \ref{T_LeVeque} reduces to a special case of (ii). In this case our
Proposition \ref{T_super_simple} below gives the same bound for $\h(x)$ as \eqref{eq:bound_m'}.

Comparing our result with those of Bugeaud
\cite{Bug2} and B\'erczes, Evertse and Gy\H ory \cite{BEG}
one can observe as follows. In Bugeaud \cite{Bug2} $f(X)$ is monic
with coefficients in $O_K$, $b\in O_K \setminus \{ 0\}$ and the effective bounds are
not given explicitly in terms of several parameters.
In B\'erczes, Evertse, Gy\H ory \cite{BEG} the roots of the polynomial
$f(X)\in O_S[X]$ are all simple. Further in \cite{BEG} the
bounds depend on $\h(b)$.

In Theorem \ref{T_LeVeque} we give completely explicit upper bounds for the
solutions in full generality. Our bounds depend on $N_S(b)$, but not on $\h(b)$. 
Hence, if $b \in O_S^*$, then $N_S(b)=1$, and so 
our bounds are independent of $b$. However, in the less general cases mentioned above
the explicitly given parts of the bounds of \cite{BEG} and \cite{Bug2} are better
in some parameters. 

Finally, we note that in the proof of \eqref{eq:bound32} we need to reduce equation
\eqref{eq:super} first to an equation of similar form over some finite extension of
$K$, where the corresponding polynomial has only simple roots, and then
over a further finite extension to a binomial Thue equation to which our
Proposition \ref{P_Thue} applies.

For convenience, in the proofs we shall frequently use the absolute logarithmic Weil height,
while in the bounds the absolute multiplicative height of $f$ and $f^*$.

\subsection{Bounding the exponent}

The following theorem is the main result of this section, which answers an extension of \cite[Exercise 4.1]{BuEMS},
 and generalises and improves~\cite[Theorem~2.3]{BEG} and \cite[Lemma 2.8]{BOSS}.
 Note that in \cite[Theorem~2.3]{BEG} the polynomial $f$ there is assumed to only have simple roots.

\begin{theorem}
\label{thm:GenS-T}
Assume that the equation \eqref{eq:super} 
has a solution $(x,y)$ with $y\ne 0$ and $y \notin O_S^*$.
Then
$$
m \leq  2C \log C,
$$
where
$$
C= 4^{12n^2 s}(10n^2 s)^{38ns} H(f)^{12nd} |D_K|^{6n} P_S^{n^2} (\log^* P_S)^{3ns}\log^*N_S(b).
$$
\end{theorem}

We remark that, compared to~\cite[Theorem~2.3]{BEG}, the upper bound for the exponent $m$ in Theorem~\ref{thm:GenS-T} depends now on the $S$-norm of $b$ rather than its height. 
This is crucial for its application in \cite{BA24}, because  if $b \in O_S^*$
then our bound does not depend on $b$. 

More precisely, the upper bound is linear in $\log^*(N_S(b)) \log^*\log^*N_S(b)$.
When $S$ contains only infinite places, one can improve this dependence
and get an upper bound linear in $\log^*N_S(b)$,
by using the third lower bound for linear forms in logarithms in \cite[Theorem 2.1]{BuEMS} (which follows from Waldschmidt's lower bound \cite[Theorem 9.1]{Wald});
see Remark~\ref{rem:improve}.

We also remark that the assumption $y \notin O_S^*$ is very natural;
see Lemma~\ref{lem:fact} below (where the assumption is used when referring to \eqref{eq:super_ideal}) and also see the theorems in \cite[Chapter 10]{ShT}.

In fact, when $y \in O_S^*$, we can find a fundamental system of $S$-units, say $\{\eta_1, \ldots, \eta_k\}$, to express $y^m = \prod_{i=1}^{k} \eta_i^{3q_i + u_i}$ with integer $q_i$ and $u_i \in \{0,1, 2\}$, and then the equation \eqref{eq:super} reduces to finitely many equations of the form $f(x) = b^{\prime}y^3$, and so there is no need to bound $m$.



\section{Auxiliary results}

\subsection{Some facts about $S$-integers} 
We recall some facts about $S$-integers (see \cite[Section 3.2]{BEG}). 

Recall that $K$ is a number field, and $S$ is a finite set of places of $K$ containing all infinite places. 
For any non-zero fractional ideal $\mathfrak{I}$ of the ring of $S$-integers $O_S$, there exists a unique fractional ideal, say $\mathfrak{I}^*$, of $O_K$ composed of prime ideals outside $S$ such that $\mathfrak{I} = \mathfrak{I}^* O_S$. Then, we define the $S$-norm of $\mathfrak{I}$ as: 
$$
N_S(\mathfrak{I}) = N_{K}(\mathfrak{I}^*). 
$$
For $\alpha \in K$, the $S$-norm $N_S(\alpha)$ of $\alpha$ defined in Section~\ref{sec:notation} is exactly equal to the $S$-norm of the fractional ideal generated by $\alpha$. 

Let $L$ be a finite extension of $K$, and $T$ the set of places of $L$ lying above the
places in $S$. Then the ring of $T$-integers $O_T$ is the integral closure in $L$ of $O_S$.
In addition, the notation $N_T(b)$ and $N_T(\mathfrak{I})$ for $b \in L$ and fractional ideal $\mathfrak{I}$ of $O_T$ 
will appear later on, their meanings follow the above paragraph.

\subsection{Some basic estimates}

Recall that $K$ is a number field of degree $d$. 
We first present some basic estimates about the regulator $R_K$ and class number $h_K$ of $K$. 

\begin{lemma}\label{L_hR}
We have: 
\begin{align}
&R_K \geq 0.2, \label{regulator_lower_bound}\\
&h_KR_K \leq |D_K|^{\frac{1}{2}}(\log^* |D_K|)^{d-1} \label{bound_hkRk}.
\end{align}
\end{lemma}

\begin{proof}
Inequality \eqref{regulator_lower_bound} is a result of Friedman \cite[Theorem B]{Friedman1}. Inequality \eqref{bound_hkRk} follows from (2) in Louboutin \cite{Louboutin1},
see also (59) in Gy\H ory and Yu \cite{GY}.
\end{proof}

We now restate some estimates about discriminants of number fields. 

\begin{lemma}\label{L_disc_I}
Let $f(X)\in K[X]$ be a polynomial of degree $n$ and discriminant $D(f)\ne 0$. Suppose that $f$ factorizes over an extension of $K$ as
$a_0(X-\al_1)\cdots(X-\al_n)$ and let $L=K(\al_1, \dots, \al_k)$
for some $k\leq n$. Then for the
discriminant $D_L$ of $L$ we have
$$
|D_L|\leq \left( n\cdot H(f) \right)^{2dkn^k} \cdot |D_K|^{n^k}.
$$
For the case $k=1$ we have the sharper estimate
$$
|D_L|\leq n^{(2n-1)d} \cdot H(f)^{(2n-2)d} \cdot |D_K|^{[L:K]}.
$$
\end{lemma}

\begin{proof}
This is Lemma 4.1 of \cite{BEG}.
\end{proof}

We can extend the result in Lemma~\ref{L_disc_I} to the case when  $D(f) = 0$. 
We first make a preparation. 

\begin{lemma}\label{lem:f*}
Let $f(X)=a_0\cdot \prod_{i=1}^r (X-\al_i)^{e_i}\in K[X]$ be a polynomial of degree $n$, where $\al_1, \dots, \al_r$ are all the distinct zeros of $f$ and for $i=1, \dots ,r$, $e_i \geq 1$. Then for the polynomial $f^*(X)$ defined by
$$
f^*(X) = a_0 \cdot \prod_{i=1}^r (X-\al_i),
$$
we have $f^*(X)\in K[X]$ and
$$
H(f^*) \leq 2^n H(f)^2.
$$
\end{lemma}

\begin{proof}
The statement $f^*(X)\in K[X]$ is trivial.
Indeed, since $f(X)\in K[X]$, the minimal polynomial of any root $\al$ of $f$ over $K$
divides $f$, thus all its roots are also roots of $f^*$.
Thus $\frac{1}{a_0}f^*$ must be the product of such minimal polynomials,
which are all monic and in $K[X]$.

The second statement is a simple consequence of the properties of the heights of polynomials.
Indeed, using \cite[Corollary 3.5.4]{EG2} about the homogeneous height $\hh$, we have
\begin{align*}
\h(f^*) & \le \h(a_0) + \h\left(\frac{1}{a_0}f^*\right) = \h(a_0) + \hh\left(\frac{1}{a_0}f^*\right) \\
 & \le  \h(a_0) + \hh(f) + n \log 2  \\
 & \le 2\h(f)+n\log 2.
\end{align*}
This in fact completes the proof. 
\end{proof}

We now state a general version of Lemma~\ref{L_disc_I}. 

\begin{lemma}\label{lem:discr_of_ext}
Let $f(X)=a_0\cdot \prod_{i=1}^r (X-\al_i)^{e_i}\in K[X]$ be a polynomial of degree $n$, where $\al_1, \dots, \al_r$ are all the distinct zeros of $f$ and for $i=1, \dots ,r$, $e_i \geq 1$.  For any fixed integer $k$ with $1 \le k \le r$,
let $L$ be the field defined by $L=K(\al_1, \dots, \al_k)$. Then, for the discriminant
of the field $L$ we have
$$
|D_L| \le (2^{n}n\cdot H(f))^{2dkn^k}  |D_K|^{n^k}.
$$
For the case $k=1$, we have the sharper estimate
$$
|D_L| \le 2^{(2n-2)nd} n^{(2n-1)d} H(f)^{(2n-2)d} |D_K|^{[L:K]}.
$$
\end{lemma}

\begin{proof}
Let $g(X) = \prod_{i=1}^r (X-\al_i)$.
As in proving Lemma~\ref{lem:f*}, we have that $g \in K[X]$ and
$$
H(g) \le 2^n H(f).
$$
Then, applying Lemma \ref{L_disc_I} to the polynomial $g$, after a short computation
we obtain the statement of the lemma.
\end{proof}

Now, let $L$ be a finite extension of $K$ and denote by $\fd_{L/K}$ the
relative discriminant ideal of $L/K$.
Recall that $\fd_{L/K}$ is the
ideal of $O_K$ (the ring of integers of $K$) generated 
by the determinants of the matrices 
$$
\Big( \Tr_{L/K}(\omega_i \omega_j) \Big)_{1\le i, j \le k}\ \ \mbox{with any } \omega_1, \dots ,\omega_k \in O_L ,
$$
where $\Tr_{L/K}$ is the trace map from $L$ to $K$ and $k=[L:K]$.
Let us denote by $\cP(K)$ the set of all non-zero prime ideals of $O_K$.


\begin{lemma}\label{L_gen_disc_primedivisors}
Let $k=[L:K]$. Then, 
for every prime ideal $\fp \in \cP(K)$ with $\ord_{\fp}(\fd_{L/K})>0$ we have
$$
\ord_{\fp}(\fd_{L/K})\leq k\cdot (1+\ord_{\fp}(u(k))),
$$
where $u(k)$ is the least common multiple $\mathrm{lcm}(1, \ldots ,k)$ of the integers $1, \ldots, k$.
\end{lemma}

\begin{proof}
This is Lemma 3.2 of \cite{BEG}.
\end{proof}

\begin{lemma}\label{L_gen_disc_extensions}
{\rm (i)} Let $M\supseteq L \supseteq K$ be a tower of finite extensions. Then we have
$$
\fd_{M/K} = N_{L/K}(\fd_{M/L}) \fd_{L/K}^{[M:L]},
$$
where $N_{L/K}$ is the norm map from $L$ to $K$. 

{\rm (ii)} Let $L_1, L_2$ be finite extensions of $K$.
Then for their compositum $L_1 L_2$ we have
$$
\fd_{L_1L_2/K} \supseteq \fd_{L_1/K}^{[L_1L_2:L_1]} \fd_{L_2/K}^{[L_1L_2:L_2]}.
$$
\end{lemma}

\begin{proof}
These are classical results which may be found
in most textbooks. In this form it is stated in Lemma 3.3 of \cite{BEG} (see also Neukirch \cite[p. 213, Korollar 2.10]{Neukirch1},
and Stark \cite[Lemma 6]{Stark1}).
\end{proof}

\begin{lemma}\label{L_gen_disc_primes_excluded}
Let $k$ be a positive integer, $\ga \in K^*$ and $L=K(\sqrt[k]{\ga})$. Further, let $\fp \in \cP(K)$ be a prime ideal with
$$
\ord_{\fp}(k)=0 \ \ \ \text{and} \ \ \  \ord_{\fp}(\ga) \equiv 0 \pmod k.
$$
Then $L/K$ is unramified at $\fp$, i.e.
$$
\ord_{\fp}(\fd_{L/K}) = 0.
$$
\end{lemma}

\begin{proof}
This is Lemma 3.4 of \cite{BEG}.
\end{proof}

\subsection{Heights}

We will frequently use the inequalities about the absolute logarithmic Weil height
$$
\h(\alpha_1 \cdots \alpha_n) \le \sum_{i=1}^{n} \h(\alpha_i), 
\qquad 
\h(\alpha_1 + \cdots + \alpha_n) \le \sum_{i=1}^{n} \h(\alpha_i) + \log n
$$
for $\alpha_1, \ldots, \alpha_n \in \overline{\Q}$ and the equality 
$$
\h(\alpha^k) = |k| \h(\alpha), \quad \alpha \in \overline{\Q}^*, k \in \Z. 
$$

Now we recall some further results about height functions. 

\begin{lemma}\label{L_lower bound_for hight}
Let $\al$ be a non-zero algebraic number of degree $d$ which is not a root of unity. Then
$$
\h(\al) \geq V(d):=
\begin{cases}
\log 2 &\text{if} \quad d = 1, \\
2/(d(\log 3d)^3) &\text{if} \quad d \geq 2.
\end{cases}
$$
\end{lemma}

\begin{proof}
See Voutier \cite[Corollary 2]{Voutier1}.
\end{proof}

Recall that $S$ is a finite set of places of $K$,
which consists of all the infinite places of $K$, 
and $O_S$ is the ring of $S$-integers of $K$. 

\begin{lemma}\label{L_bound_hight_by_norm}
For every $\al\in O_S \setminus \{ 0 \}$ and every positive integer $k$, there exists an $S$-unit $\eta \in O_S^*$ with
$$
\h(\al\eta^k)\leq \frac{1}{d}\log N_S(\al) + k \left( cR_K+\frac{h_K}{d}
\log Q_S \right),
$$
where $c=39d^{d+2}$ and $Q_S$ has been defined in Section~\ref{sec:notation}.
\end{lemma}

\begin{proof}
This is a slightly weaker version of Lemma 3 of Gy\H ory and Yu \cite{GY}.
In fact, the result was essentially proved (with a larger constant)
in \cite{BugGy2} and \cite{Gy21}.
\end{proof}

\begin{lemma}\label{L_bound_roots_by_pol}
Let $\alpha_1, \dots ,\alpha_n\in\OQq$ and $f=(X-\alpha_1)\cdots (X-\alpha_n)$.
Then
\[
|\h(f)-\sum_{i=1}^n \h(\alpha_i)|\leq n\log 2.
\]
\end{lemma}

\begin{proof}
This is a special case of Bombieri and Gubler \cite[p.28, Theorem 1.6.13]{BombGub}, stated in this form in Lemma 3.6 in \cite{BEG}.
\end{proof}

\begin{lemma}\label{disc-height}
Let $f=a_0X^n+a_1X^{n-1}+\cdots +a_n\in K[X]$ be 
a polynomial of degree $n$ with discriminant $D(f)\not= 0$.
Then
\begin{eqnarray*}
{\rm (i)}&&|D(f)|_v\leq n^{(2n-1)s(v)}\max (|a_0|_v, \dots ,  |a_n|_v)^{2n-2}\ \ \mbox{for } v\in M_K,
\\
{\rm (ii)}&& \h(D(f))\leq (2n-1)\log n +(2n-2)\h(f),
\end{eqnarray*}
where $s(v)=1$ if $v$ is real, $s(v)=2$ if $v$ is complex, $s(v)=0$ if $v$ is finite.
\end{lemma}

\begin{proof}
This is Lemma 3.7 of \cite{BEG}.
\end{proof}

\subsection{Baker's method}

Let $n \ge 2$ be an integer, 
$\al_1, \dots ,\al_n$ be $n$ non-zero elements of $K$, and $b_1,\dots ,b_n$ be rational integers, not all zero. 
Recall that $M_K$ is the set of places of $K$.

 Put
\begin{eqnarray*}
&&\Lambda = \al_1^{b_1} \cdots \al_n^{b_n}-1,
\\
&&\Theta =\prod_{i=1}^n \max \big( h(\alpha_i), V(d)\big),
\\
&&B=\max (3,|b_1|,\dots ,|b_n|),
\end{eqnarray*}
where $V(d)$ is the lower bound from Lemma \ref{L_lower bound_for hight}, and
write for $v \in M_K$
$$
N(v) =
\begin{cases}
2 & \text{if $v$ is infinite} \\
N_K(\fp) & \text{if $v=\fp$ is finite}.
\end{cases}
$$

\begin{proposition}\label{P_Baker}
Suppose that $\Lambda \ne 0$. Then for $v\in M_K$ we have
\begin{equation}\label{lowerbound-Lambda}
\log |\Lambda|_v > -\, c_1(n, d)\frac{N(v)}{\log N(v)} \Theta\log B ,
\end{equation}
where $c_1(n,d) =  12(16ed)^{3n+2} (\log^* d)^2.$
\end{proposition}

\begin{proof}
This is Proposition 3.10 of \cite{BEG}, which is a simple consequence of
a result of Matveev \cite[Corollary 2.3]{Matveev1} and of
K. Yu \cite{Yu1} (consequence of Main Theorem on p. 190).
\end{proof}

\subsection{Thue equations and Pell equations}\label{thue-pell}

We recall some notations introduced in Section \ref{sec:notation}. 
Namely, let $K$ be a number field of degree $d$. 
Let $D_K$ be the discriminant, $R_K$ the regulator and $h_K$ the class number of $K$, respectively. Let $S$ be a finite set of places of $K$ containing all infinite places. Put $s:=\# S$ and denote by $O_S$ the ring of $S$ integers in $K$, and by $R_S$ the $S$-regulator. 
Let $\fp_1\dots\fp_{s^\prime}$ be the prime ideals in $S$, and put
$$
P_S=\max\{ N_K(\fp_1),\dots, N_K(\fp_{s^\prime})\},\ \ \ Q_S=N_K(\fp_1\cdots \fp_{s^\prime}),
$$
with the convention that $P_S=Q_S=1$ if $S$ does not contain any finite places.

The below effective results on Thue equations and on systems of Pell equations
are easy consequences of a general effective result on decomposable form
equations by Gy\H{o}ry and Yu \cite{GY}. We will use the constant
\[
c_2(s,d)= s^{2s+4}2^{7s+60}d^{2s+d+2}.
\]

\begin{proposition}\label{P_Thue}
Let $\beta\in K^*$ and let $F(X,Y)=\sum_{i=0}^n a_iX^{n-i}Y^i\in K[X,Y]$ be a binary form of degree $n\geq 3$
with non-zero discriminant which splits into linear factors over $K$.
Suppose that
\[
\max_{0\leq i\leq n} \h(a_i)\leq A,\ \ \ \h(\beta )\leq B.
\]
Then for the solutions of
\begin{equation}\label{Thue}
F(x,y)=\beta\ \ \ \mbox{in } x,y\in O_S, 
\end{equation}
we have
\begin{eqnarray}
\label{Thue-bound}
&&\max (\h(x),\h(y))
\\
\nonumber
&&
\leq c_2(s,d)n^6P_SR_S\left( 1+\frac{\log^* R_S}{\log^* P_S}\right)\cdot \Big( R_K+\frac{h_K}{d}\log Q_S+ndA+B\Big).
\end{eqnarray}
\end{proposition}

\begin{proof}
This is a result of Gy\H{o}ry and Yu \cite[p. 16, Corollary 3]{GY}.
For this (slightly weaker) version see \cite[Proposition 3.11]{BEG}.
\end{proof}

\begin{proposition}\label{P_Pell}
Let $\gamma_1,\gamma_2,\gamma_3,\beta_{12},\beta_{13}$ be non-zero elements of $K$
such that
\begin{eqnarray*}
&\beta_{12}\not=\beta_{13},\ \ \
\sqrt{\gamma_1/\gamma_2},\, \sqrt{\gamma_1/\gamma_3}\in K,&
\\
&\h(\gamma_i)\leq A\ \mbox{for } i=1,2,3,\ \ \ \h(\beta_{12}),\h(\beta_{13})\leq B.
\end{eqnarray*}
Then for the solutions of the system
\begin{equation}\label{Pell}
\gamma_1x_1^2-\gamma_2x_2^2=\beta_{12},\ \ \gamma_1x_1^2-\gamma_3x_3^2=\beta_{13}\ \ \ \
\mbox{in } x_1,x_2,x_3\in O_S, 
\end{equation}
we have
\begin{eqnarray}
\label{Pell-bound}
&&\max (\h(x_1),\h(x_2),\h(x_3))
\\
\nonumber
&&
\leq c_2(s,d)P_SR_S\left( 1+\frac{\log^* R_S}{\log^* P_S}\right)
\cdot \Big( R_K+\frac{h_K}{d}\log Q_S+dA+B\Big).
\end{eqnarray}
\end{proposition}

\begin{proof}
This is Proposition 3.12 of \cite{BEG}.
\end{proof}

\section{Preparations for proving Theorem \ref{T_LeVeque}}

We keep the notation in Section~\ref{sec:notation} and introduce some further notation which will be needed for the proofs. For $i=1, \ldots, r$, let $L_i=K (\alpha_i)$, and by $d_{L_i},h_{L_i}, D_{L_i}, R_{L_i}$ we denote
the degree, class number, discriminant and regulator of $L_i$. Further, let $T_i$ be the set of places of $L_i$
lying above the places in $S$, $O_{T_i}$ the ring of $T_i$-integers in $L_i$, $R_{T_i}$ the $T_i$-regulator of $L_i$ (see (3.8) of \cite{BEG}), and $t_i$ the cardinality of $T_i$.
We denote by $[\beta_1,\dots , \beta_k]_{T_i}$ the fractional ideal of $O_{T_i}$
generated by  $\beta_1,\dots , \beta_k$.


The following lemma will be crucial for the proof of Theorem \ref{T_LeVeque}.

\begin{lemma} \label{L_crucial}
Let $x,y \in O_S$ be a solution of the equation \eqref{eq:super}
with $y \neq 0$. Then, for $i=1,...,r$ we have the following:
\newline 
{\rm (i)} there are ideals $\mathfrak{C}_i, \mathfrak{A}_i$ of $O_{T_i}$ such that
$$
[a_0(x-\alpha_i)]_{T_i}=\mathfrak{C}_i\mathfrak{A}_i^{m_i}, \qquad \mathfrak{C}_i \supset [a_0bD(f^*)]_{T_i}^{m_i-1}; 
$$
{\rm (ii)} there are $\de_{i1}, \de_{i2}, \xi_i$ with
$$
\begin{cases}
a_0(x-\alpha_i)=(\de_{i1}/\de_{i2})\cdot \xi_i^{m_i}, \quad \de_{i1}, \de_{i2}, \xi_i \in O_{T_i}, \\ \h(\de_{i2}), \h(\de_{i1})\leq
m_i(dr^3 H(f^*)^2)^{dr}  |D_{K}|^r \left(80(dr)^{dr+2}+ \log Q_S \right) +m_i \log N_S(b).
\end{cases}
$$
\end{lemma}

For earlier versions of Lemma \ref{L_crucial} see e.g. Bugeaud \cite{Bug2} and
B\'erczes, Evertse and Gy\H ory \cite[Lemma 4.2]{BEG}.

\begin{proof}
 It suffices to prove the lemma for $i=1$. We put
 $$
 g_1(X)=\dfrac{f(X)}{a_0(X-\alpha_1)^{e_1}}.
 $$
 Let $M$ denote the splitting field of $f$ over $K$.
 By $[\cdot]$ we denote the fractional ideal in $M$ with respect to the integral closure of $O_{T_1}$ in $M$. Also, for any polynomial $g \in M[X]$ let us denote by $[g]$ the fractional ideal generated by the coefficients of $g$. By definition, we have
 $$
 [g_1(x)]=\prod_{i=2}^r[x-\alpha_i]^{e_i}, \qquad [g_1]=\prod_{i=2}^r[1,\alpha_i]^{e_i},
 $$
 and
 $$
 [f^*]=[a_0]\prod_{i=1}^r[1,\alpha_i], \qquad [D(f^*)]=[a_0]^{2r-2}\prod_{i=1}^r\prod_{j \neq i} [\alpha_i-\alpha_j],
 $$
where $f^*(X) = a_0 \prod_{i=1}^{r} (X - \alpha_i)$. 

For any $1 \le i \ne j \le r$, since $[\alpha_j - \alpha_i] \subseteq [1,\alpha_j][1,\alpha_i]$, the following fractional ideal 
$$
  \frac{[\alpha_j - \alpha_i]}{[1,\alpha_j][1,\alpha_i]}
$$
is in fact an integral ideal. 
Clearly,  we have 
$$
\frac{[x - \alpha_1]}{[1,\alpha_1]} + \frac{[x - \alpha_i]}{[1,\alpha_i]} 
\supseteq  \frac{[\alpha_1 - \alpha_i]}{[1,\alpha_1][1,\alpha_i]}
$$
for $i=2, \ldots, r$, and so 
$$
\frac{[x - \alpha_1]}{[1,\alpha_1]} + \frac{[x - \alpha_i]^{e_i}}{[1,\alpha_i]^{e_i}} 
\supseteq  \frac{[\alpha_1 - \alpha_i]^{e_i}}{[1,\alpha_1]^{e_i}[1,\alpha_i]^{e_i}}. 
$$ 
Then, we obtain 
$$
\frac{[x - \alpha_1]}{[1,\alpha_1]} + \prod_{i=2}^{r}\frac{[x - \alpha_i]^{e_i}}{[1,\alpha_i]^{e_i}} 
\supseteq  \prod_{i=2}^{r}  \frac{[\alpha_1 - \alpha_i]^{e_i}}{[1,\alpha_1]^{e_i}[1,\alpha_i]^{e_i}}, 
$$
and thus 
\begin{align*}
\frac{[x - \alpha_1]^{e_1}}{[1,\alpha_1]^{e_1}} + \prod_{i=2}^{r}\frac{[x - \alpha_i]^{e_i}}{[1,\alpha_i]^{e_i}} 
& \supseteq  \prod_{i=2}^{r}  \frac{[\alpha_1 - \alpha_i]^{e_i e_1}}{[1,\alpha_1]^{e_i e_1}[1,\alpha_i]^{e_i e_1}} \\
& \supseteq  \left( \prod_{i=2}^{r}  \frac{[\alpha_1 - \alpha_i]}{[1,\alpha_1][1,\alpha_i]} \right)^{(n-e_1)e_1}, 
\end{align*}
where the right-hand side contains 
$$
\left( \prod_{j=1}^{r}\prod_{i \ne j} \dfrac{[\alpha_j - \alpha_i]}{[1,\alpha_j][1,\alpha_i]} \right)^{(n-e_1)e_1} = \left( \dfrac{[D(f^*)]}{[f^*]^{2r - 2}} \right)^{e_1(n-e_1)} 
\supseteq  \frac{[D(f^*)]^{e_1(n-e_1)}}{[f^*]^{2n - 2}}, 
$$
where the last inclusion comes from $2n - 2 \le (2r-2)e_1 (n-e_1)$.
So, we get 
 \begin{equation}\label{eq:factor_gcd}
 \dfrac{[x-\alpha_1]^{e_1}}{[1,\alpha_1]^{e_1}}+ \dfrac{[g_1(x)]}{[g_1]} \supseteq \dfrac{[D(f^*)]^{e_1(n-e_1)}}{[f^*]^{2n - 2}}.
 \end{equation}
 Writing the equation \eqref{eq:super} as equation of ideals and noticing $[f]=[a_0]\prod_{i=1}^{r}[1,\alpha_i]^{e_i}$, we have
 \begin{equation}\label{eq:super_ideal}
     [b][f]^{-1}[y]^m=\dfrac{[x-\alpha_1]^{e_1}}{[1,\alpha_1]^{e_1}}\cdot \dfrac{[g_1(x)]}{[g_1]}.
 \end{equation}
 Note that the ideals occurring above are all defined over $L_1$, so we can view them as fractional ideals of $O_{T_1}.$ Thus, from now on, we use $[\cdot]$ to denote fractional ideals of $O_{T_1}.$

 Let $\mathfrak{P}$ be a prime ideal of $O_{T_1}$ not dividing $a_0bD(f^*)$.
 Since $D(f^*) \in [f^*]^{2r-2}$ (due to the determinantal expression of $D(f^*)$),
 $\mathfrak{P}$ divides neither $[f^*]$ nor $[f]$.
 Hence, by \eqref{eq:factor_gcd}, the ideal $\mathfrak{P}$ divides at most one of the ideals $\dfrac{[x-\alpha_1]^{e_1}}{[1,\alpha_1]^{e_1}}$ and $\dfrac{[g_1(x)]}{[g_1]}$, so using \eqref{eq:super_ideal} we obtain
 $$
 \text{ord}_\mathfrak{P} \dfrac{[x-\alpha_1]^{e_1}}{[1,\alpha_1]^{e_1}} \equiv 0   \quad (\text{mod } m),
 $$
 whence
 $$
 \text{ord}_\mathfrak{P} \dfrac{[x-\alpha_1]}{[1,\alpha_1]} \equiv 0   \quad (\text{mod } m_1),
 $$
 where $m_1=m/\gcd(m,e_1)$. The ideal $[a_0][1,\alpha_1]$ is not divisible by $\mathfrak{P}$ since it contains  $a_0$. Hence,
 $$
 \text{ord}_\mathfrak{P} \left(a_0[x-\alpha_1]\right)\equiv 0   \quad (\text{mod } m_1).
 $$
 Applying division with remainder to the exponents of the prime ideals dividing $a_0bD(f^*)$ in the factorization of $a_0[x-\alpha_1]$, we obtain that there are ideals $\mathfrak{C}_1, \mathfrak{A}_1$ of $O_{T_1}$, with $\mathfrak{C}_1$ dividing $\left(a_0bD(f^*)\right)^{m_1-1}$ such that $[a_0(x-\alpha_1)]=\mathfrak{C}_1 \mathfrak{A}_1^{m_1}.$ This proves the statement (i) of the lemma.

 Now, $\mathfrak{A}_1$ can be written as $\mathfrak{A}_1=\mathfrak{A}_1^*O_{T_1}$ with an ideal $\mathfrak{A}_1^*$ of $O_{L_1}$ composed of prime ideals outside $T_1$, and moreover, we can choose a non-zero $\widetilde{\xi}_1 \in \mathfrak{A}_1^*$ with $|N_{L_1/\Q}(\widetilde{\xi}_1)|\leq |D_{L_1}|^{1/2}N_{L_1}(\mathfrak{A}_1^*)$
 (see Lang \cite[pp. 119/120]{Lang70}). This implies $N_{T_1}(\widetilde{\xi}_1)\leq |D_{L_1}|^{1/2}N_{T_1}(\mathfrak{A}_1)$, i.e., $[\widetilde{\xi}_1]=\mathfrak{B}_1\mathfrak{A}_1$ where $\mathfrak{B}_1$ is an ideal of $O_{T_1}$ with $N_{T_1}(\mathfrak{B}_1)\leq |D_{L_1}|^{1/2}$. Similarly, there exists $\widetilde{\delta_{11}} \in L_1$ such that $[\widetilde{\delta_{11}}]=\mathfrak{D}_1 \mathfrak{C}_1$, where $\mathfrak{D}_1$ is an ideal of $O_{T_1}$ with $N_{T_1}(\mathfrak{D}_1)\leq |D_{L_1}|^{1/2}.$ 

So, we have 
$$
[a_0(x-\alpha_1)] = \frac{[\widetilde{\delta_{11}} \widetilde{\xi}_1^{m_1}]}{\mathfrak{D}_1\mathfrak{B}_1^{m_1}}.
$$
This implies that
 $$
 a_0(x-\alpha_1)=\dfrac{\widetilde{\delta_{11}}}{\widetilde{\delta_{12}}}\widetilde{\xi}_1^{m_1},
 $$
 where $\widetilde{\delta_{11}}, \widetilde{\delta_{12}} \in O_{T_1}$ and
 $$
 [\widetilde{\delta_{12}}]=\mathfrak{D}_1\mathfrak{B}_1^{m_1}.
 $$
 Using (i) and the choices of $\mathfrak{B}_1$ and $\mathfrak{D}_1$, we obtain
 \begin{equation}\label{eq:norm_estimate}
     N_{T_1}(\widetilde{\delta_{11}})\leq |D_{L_1}|^{1/2}N_{T_1}(ba_0D(f^*))^{m_1-1}, \quad N_{T_1}(\widetilde{\delta_{12}})\leq |D_{L_1}|^{(m_1+1)/2}.
 \end{equation}
 By Lemma \ref{L_bound_hight_by_norm} we can find $T_1$-units $\eta_1,\eta_2 \in O^*_{T_1}$ such that $\de_{1j}:=\widetilde{\delta_{1j}}\eta_j^{m_1}$ for $j=1,2$, and
 \begin{equation}
 \h(\de_{1j})\leq d_{L_1}^{-1}\log N_{T_1}(\widetilde{\delta_{1j}})+m_1 \cdot \left(cR_{L_1}+ \dfrac{h_{L_1}}{d_{L_1}}\log Q_{T_1} \right) \text{ for } j=1,2, 
 \end{equation}
 where $d_{L_1}=[L_1:\mathbb{Q}], c=39d_{L_1}^{d_{L_1}+2}$ and $Q_{T_1}=\displaystyle\prod_{\substack{\mathfrak{P}\in T_1,\\ \mathfrak{P} \text{ finite}}}N_{L_1} (\mathfrak{P}).$
 Making $\xi_1=\eta_2\eta_1^{-1}\widetilde{\xi}_1$,
 we have 
$$a_0(x-\alpha_1)=(\de_{11}/\de_{12})\cdot \xi_1^{m_1},$$
 where $\xi_1 \in O_{T_1}$.

 Due to $a_0bD(f^*) \in K$, we have
    \begin{equation}
         \begin{split}
             d_{L_1}^{-1} &\log N_{T_1}(a_0bD(f^*))=d^{-1}\log N_S(a_0bD(f^*))\\
             &\leq \h(a_0D(f^*))+ \log N_S(b) \\
             &\leq \h(a_0) + \h(D(f^*)) + \log N_S(b) \\
             &\leq \h(f^*) +\h(D(f^*))+ \log N_S(b).
         \end{split}
     \end{equation}
 Together with Lemma \ref{L_disc_I},  Lemma \ref{disc-height} and \eqref{eq:norm_estimate}, this implies
 \begin{equation}\label{eq:estimate_3}
     \begin{split}
         d_{L_1}^{-1} \log N_{T_1}(\widetilde{\de_{1j}}) & \leq d_{L_1}^{-1} \left(\dfrac{m_1 + 1}{2}\log|D_{L_1}|+m_1\log N_{T_1}(a_0bD(f^*)) \right)\\
         &\leq m_1 \big( 4r \h(f^*) +4r^2+ \log |D_K|+\log N_S (b) \big), \\
     \end{split}
 \end{equation}
for $j=1,2$.
 Next, by Lemma~\ref{L_hR} and Lemma~\ref{L_disc_I} and noticing $d_{L_1}\leq dr$ and $r \ge 2$, we have
 \begin{equation}\label{eq:hR_estimate}
 \begin{split}
 \max (h_{L_1},R_{L_1}) & \leq 5 |D_{L_1}|^{1/2} (\log^* |D_{L_1}|)^{dr-1} \le (dr)^{dr}|D_{L_1}| \\
& \leq (d r^3)^{dr} H(f^*)^{(2r-2)d} |D_{K}|^r, 
 \end{split}
 \end{equation}
where the second inequality requires $|D_{L_1}| > 1$ (this condition means $|D_{L_1}| \ge 4$ due to Stickelberger's theorem) and this inequality can be proved by fixing $rd$ and viewing $|D_{L_1}|$ as a variable; 
and for the case $|D_{L_1}| = 1$ (that is, $L_1 = \Q$) 
the final inequality about $\max (h_{L_1},R_{L_1})$ still holds due to $h_{L_1} = R_{L_1} = 1$ in this case. 

 Together with $d^{-1}_{L_1} \log Q_{T_1}\leq d^{-1} \log Q_S$,
 \eqref{eq:norm_estimate}-\eqref{eq:hR_estimate} give the desired bound in (ii).
\end{proof}

Using the above notation, let
\begin{equation}\label{def_ga_i}
\ga_i:=a_0^{-1}\cdot\frac{\de_{i1}}{\de_{i2}},
\end{equation}
whence
$$
x-\al_i=\ga_i\xi_i^{m_i}, \quad\quad\quad i=1, \ldots, r.
$$

\begin{lemma}\label{L_field_discriminant_bound}
{\rm (i)} Let $m \geq 3$ and $M=K(\alpha_1,\alpha_2, \sqrt[m']{\gamma_1/\gamma_2}, \rho)$, where $m'=\gcd(m_1, m_2)$ and $\rho$ is a primitive $m'$-th root of unity. Then
$$
|D_M| \leq 10^{dm^3r^2}r^{4dm^2r^3}|D_K|^{m^2r^2}H(f^*)^{4dm^2r^3}Q_S^{m^2r^2}N_S(b)^{m^2r^2}.
$$
{\rm (ii)} Suppose $m_1=m_2=m_3=2$ and let
$M=K(\al_1,\al_2,\al_3,\sqrt{\ga_1/\ga_2},\sqrt{\ga_1/\ga_3})$. Then
$$
|D_M| \leq r^{40dr^4}|D_K|^{4r^3}H(f^*)^{25dr^4}Q_S^{8r^3}N_S(b)^{8r^3}.
$$
{\rm (iii)} Let $m \geq 3$ and $M=K(\alpha_1,\alpha_2, \sqrt[m_1]{\gamma_1},\sqrt[m_2]{\gamma_2}, \rho)$ where $\rho$ is a primitive $m_1m_2$-th root of unity. Then
$$
|D_M| \leq 2^{10dm^5r^2}r^{8dm^4r^3}|D_K|^{m^4r^2}H(f^*)^{8dm^4r^3}Q_S^{2m^4r^2}N_S(b)^{2m^4r^2}.
$$
\end{lemma}

Statements similar to (i) and (ii) have been proved in \cite[Lemma 4.3]{BEG}, so here
we prove (iii) in detail and just point out the main steps needed to obtain the bounds
in (i) and (ii).

\begin{proof}
First, we prove (iii). 
Put $L=K(\al_1,\al_2)$, $F=L(\rho)$, $F_1=F(\sqrt[m_1]{\ga_1})$ and $F_2 = F(\sqrt[m_2]{\ga_2})$.
Using Lemma \ref{L_disc_I} with $k=2$ for the polynomial $f^*$ we obtain
\begin{equation}\label{D_L_bound}
|D_L|\leq r^{4dr^2} H(f^*)^{4dr^2} |D_K|^{r^2}; 
\end{equation}
and then applying this lemma to $F$ and the polynomial $X^{m_1m_2}-1$ over $L$ with $k=1$,  we have
\begin{equation}\label{D_M_bound}
|D_F|\leq m^{4dm^2 r^2} r^{4dm^2r^2} H(f^*)^{4dm^2r^2} |D_K|^{m^2r^2}.
\end{equation}

Now we note that $M=F_1\cdot F_2$ and by (ii) of Lemma \ref{L_gen_disc_extensions} we
obtain
$$
\fd_{M/F} \supseteq \fd_{F_1/F}^m \cdot \fd_{F_2/F}^m, 
$$
and then applying (i) of Lemma \ref{L_gen_disc_extensions} to the field extensition $M \supseteq F \supseteq \Q$,  we have
\begin{equation}\label{D_F_by_F_1_F_2}
|D_M|=N_F(\fd_{M/F})\cdot |D_F|^{[M:F]}\leq
N_F(\fd_{F_1/F})^{m}\cdot N_F(\fd_{F_2/F})^{m}\cdot |D_F|^{m^2}.
\end{equation}
So, it remains to estimate $N_F(\fd_{F_i/F})$, $i=1,2$.

Fix $i=1$ or $i=2$. Let $U$ be the set of prime ideals of $O_F$ that divide the prime
ideals from $S$ or $m_ia_0bD(f^*)$.
Let $\fP\not\in U$ be a prime ideal of $O_F$. 
Then, we directly have $\ord_{\fP}(m_i) = 0$, and by Lemma \ref{L_crucial} we have 
\[
\ord_{\fP}(\gamma_i)\equiv\ord_{\fP}\left(a_0(x-\alpha_i)\right)\equiv 0 \pmod{m_i},
\]
and so by Lemma \ref{L_gen_disc_primes_excluded}, the field extension $F_i/F$ is unramified at $\fP$, and
thus $\fd_{F_i/F}$ is composed of prime ideals from $U$. Now by Lemma \ref{L_gen_disc_primedivisors} we have
\begin{eqnarray}\label{discrelation-3}
\fd_{F_i/F}&\supseteq&\prod_{\fP\in U}\fP^{m_i(1+\ord_{\fP}(u(m_i))}
\\
\nonumber
&\supseteq&\prod_{\fP\in U}\fP^{m_i}\prod_{\fP \in U}\fP^{m_i\ord_{\fP}(u(m_i))}
\supseteq u(m_i)^{m_i}\prod_{\fP\in U}\fP^{m_i}.
\end{eqnarray}
By Rosser and Schoenfeld \cite[Corollary 1]{RosserSchoenfeld}, we have
$$
u(m_i) = \mathrm{lcm}(1, \ldots, m_i) \leq m_i^{\pi (m_i)}\leq m_i^{1.3m_i/\log m_i} \leq 4^{m_i},
$$
 where $\pi(m_i)$ is the number of prime numbers not greater than $m_i$, and thus
$|N_{F/\Q}(u(m_i)^{m_i})|\leq 4^{dm^4r^2}$.

Now let $V$ be the set of prime ideals of $O_K$ which are contained
in $S$ or divide $m_ia_0bD(f^*)$. Then, using also \eqref{eq:NS} and Lemma \ref{disc-height} we have
\begin{eqnarray*}
N_F(\prod_{\fP\in U}\fP)&\leq& N_K(\prod_{\fp\in V}\fp)^{[F:K]}
\leq N_K(\prod_{\fp\in V}\fp)^{m^2r^2}
\\
&\leq& \left(Q_S \cdot N_S(m_i a_0bD(f^*)) \right)^{m^2r^2} \\
&\leq&
Q_S^{m^2r^2}e^{d\cdot \h(m_i a_0D(f^*))m^2r^2} N_S(b)^{m^2r^2}
\\
&\leq& Q_S^{m^2r^2}m^{dm^2r^2}e^{dm^2r^2(\h(a_0)+(2r-1)\log r+(2r-2)\h(f^*))} N_S(b)^{m^2r^2}
\\
&\leq& Q_S^{m^2r^2}m^{dm^2r^2}r^{2dm^2r^3}H(f^*)^{2dm^2r^3} N_S(b)^{m^2r^2}.
\end{eqnarray*}
Consequently, for $i=1,2$ we get
$$
N_F(\fd_{F_i/F})\leq 4^{dm^4r^2}m^{dm^3r^2}r^{2dm^3r^3}H(f^*)^{2dm^3r^3} Q_S^{m^3r^2} N_S(b)^{m^3r^2}.
$$

Combining this latter estimate with \eqref{D_M_bound} and \eqref{D_F_by_F_1_F_2},  we obtain 
$$
|D_M| \le 2^{4dm^5 r^2} m^{6dm^4r^2} r^{8dm^4r^3} |D_K|^{m^4r^2} H(f^*)^{8dm^4r^3} Q_S^{2m^4r^2} N_S(b)^{2m^4r^2}, 
$$
which, together with $m \le 2^m$, gives the desired bound on $D_M$.

Now, we prove (i). 
In this case, we put $L=K(\alpha_1,\alpha_2)$,
$M_1=L(\sqrt[m']{\gamma_1/\gamma_2})$ and $M_2=L(\rho )$,  
where $m' = \gcd(m_1, m_2)$ and $\rho$ is a primitive $m'$-th root of unity. 
Then, we have again \eqref{D_L_bound} about $D_L$, and since
$M=M_1M_2$ we can use again (ii) of Lemma \ref{L_gen_disc_extensions} to
obtain
$$
\fd_{M/L} \supseteq \fd_{M_1/L}^{m'} \cdot \fd_{M_2/L}^{m'}  \supseteq \fd_{M_1/L}^{m} \cdot \fd_{M_2/L}^{m}; 
$$
and then by (i) of Lemma \ref{L_gen_disc_extensions} we have
\begin{equation*}
|D_M|=N_L(\fd_{M/L})\cdot |D_L|^{[M:L]}\leq
N_L(\fd_{M_1/L})^{m}\cdot N_L(\fd_{M_2/L})^{m}\cdot |D_L|^{m^2}.
\end{equation*}
So, it suffices to estimate $N_L(\fd_{M_i/L})$, $i=1,2$.

Here, applying \cite[Lemma 3.1]{BEG} to the field $M_2$ and the polynomial $X^{m'} - 1$, we have $\fd_{M_2/L}\supseteq [m']^{m'} \supseteq [m]^{m}$, where $[m'] = m'O_L$ and $[m] = mO_L$, and so we obtain 
$$
N_L(\fd_{M_2/L}) \le N_{L/\Q}(m)^m \le m^{dmr^2}. 
$$

In addition, since for any prime ideal $\fP$ of $O_L$ not dividing a prime ideal from $S$ and not dividing $ma_0bD(f^*)$, by Lemma~\ref{L_crucial} we have
\[
\ord_{\fP}(\gamma_1\gamma_2^{-1})\equiv\ord_{\fP}
\left(\frac{a_0(x-\alpha_1)}{a_0(x-\alpha_2)}\right)\equiv 0 \pmod{m'},
\]
and so by Lemma \ref{L_gen_disc_primes_excluded}, the field extension $M_1/L$ is unramified at $\fP$. 
Then, applying similar arguments as in proving (iii) for $N_F (\fd_{F_i/F})$, we obtain 
$$
N_L(\fd_{M_1/L})\leq 4^{dm^2r^2}m^{dmr^2}r^{2dmr^3}H(f^*)^{2dmr^3} Q_S^{mr^2} N_S(b)^{mr^2}.
$$
So, collecting the above relevant estimates and after a straightforward computation, we get the statement (i) of our Lemma.

Finally, we prove (ii). 
To prove (ii) we put $L=K(\al_1,\al_2,\al_3)$.
Then using again Lemma \ref{L_disc_I} with $k=3$ for the polynomial $f^*$ we get
\begin{equation}\label{D_L_bound_3}
|D_L|\leq r^{6dr^3} H(f^*)^{6dr^3} |D_K|^{r^3}. 
\end{equation}
Now, let $M_1=L(\sqrt{\ga_1/\ga_2})$ and $M_2=L(\sqrt{\ga_1/\ga_3})$.
Then $M=M_1 M_2$. Hence, by (ii) of Lemma \ref{L_gen_disc_extensions} we have 
$$
\fd_{M/L}\supseteq \fd_{M_1/L}^2\cdot \fd_{M_2/L}^2, 
$$
and then by (i) of Lemma \ref{L_gen_disc_extensions} we get
\begin{equation} \label{eq:DM-2}
|D_M| \leq  N_L (\fd_{M_1/L})^2 \cdot N_L (\fd_{M_2/L})^2 \cdot |D_L|^4.
\end{equation}
Then, we can estimate from above $N_L (\fd_{M_i/L})$
for $i=1,2$ as we have done in the proof of (iii) for bounding $N_F(\fd_{F_i/F})$,
and we obtain 
$$
N_L (\fd_{M_i/L}) \le 2^{4dr^3} r^{4dr^4}H(f^*)^{4dr^4} Q_S^{2r^3} N_S(b)^{2r^3}. 
$$
This, together with \eqref{D_L_bound_3}, \eqref{eq:DM-2} and $r \ge 3$, 
gives the upper bound stated in (ii) for $|D_M|$.
\end{proof}

The following proposition is about the case when the polynomial $f(X)$ only has simple zeros. 
This can be compared with Theorem 2.1 of \cite{BEG}. For $b\in O_S^*$ our bound is better, since it does not depend on $b$ in this case.

\begin{proposition}\label{T_super_simple}
Assume that $f(X)=a_nX^n+a_{n-1}X^{n-1}+\dots+a_0\in O_S[X]$ has only simple zeros and $m \geq 3, n\geq 2$. If $x,y \in O_S$ is a solution to the equation \eqref{eq:super}, then we have
$$
\h(x) \leq (6ns)^{14m^3n^3s}|D_K|^{2m^2n^2}H(f)^{8dm^2n^3}Q_S^{3m^2n^2}N_S(b)^{2m^2n^2}.
$$
\end{proposition}

\begin{proof}
We follow the proof of Theorem 2.1 of \cite{BEG}. The only difference is
that we use the upper bound for $|D_M|$ from Lemma \ref{L_field_discriminant_bound}
instead of Lemma 4.3 of \cite{BEG}. For convenience of the reader, we repeat the
other parts of the proof of Theorem 2.1 of \cite{BEG}.

Let $m\geq 3$ and $x,y \in O_S$ be a solution to $by^m=f(x)$ with $y \neq 0$. 
Since the polynomial $f(X)$ has only simple zeros, we have that $m_1=\cdots = m_n = m$. 
Then, we have $x-\alpha_i=\gamma_i\xi_i^m \quad (i=1,...,n)$ with $\gamma_i,\xi_i$ as in Lemma \ref{L_crucial} and \eqref{def_ga_i}. Let 
$$
M=K(\alpha_1,\alpha_2, \sqrt[m]{\gamma_1/\gamma_2}, \rho),
$$
 where $\rho$ is a primitive $m$-th root of unity, and let $T$ be the set of places of $M$ lying above the places from $S$. Let $\mathfrak{p}_1,...,\mathfrak{p}_{s^\prime}$ be the prime ideals (that is, finite places) in $S$, and $\mathfrak{P}_1,...,\mathfrak{P}_{t^\prime}$ be the prime ideals in $T$. Then $t^\prime \leq [M:K]s^\prime \leq m^2n^2 s^\prime.$ Further, let $P_T= \max_{i=1}^{t^\prime} N_M (\mathfrak{P}_i), Q_T= \prod_{i=1}^{t^\prime} N_M (\mathfrak{P}_i)$. We then have
\begin{equation*}
    {\gamma_1}{\xi_1}^m- {\gamma_2}{\xi_2}^m=\alpha_2-\alpha_1, \text{ with } {\xi_1}, {\xi_2} \in O_T,
\end{equation*}
and the left-hand side is a binary form of non-zero discriminant which splits into linear factors over $M$.
By Proposition \ref{P_Thue} we have
\begin{equation*}
\begin{split}
\h(\xi_1) \leq c^\prime_2 m^6P_TR_T\left(1+ \dfrac{\log^* R_T}{\log^* P_T} \right) \left(R_M+ h_Md_M^{-1}\log Q_T +md_MA+B \right),
\end{split}
\end{equation*}
where $A=\max(\h(\gamma_1), \h(\gamma_2)), B=\h(\alpha_1-\alpha_2), d_M=[M:\mathbb{Q}]$ and $c_2^\prime$ is the constant $c_2$ from Proposition \ref{P_Thue}, but with $s,d$ replaced by $m^2n^2s, m^2n^2d$ for the cardinality of $T$ and $[M:\mathbb{Q}],$ respectively, and $R_T$ is the $T$-regulator.

Using $d\leq 2s$, $m \ge 3$ and $n \ge 2$, we can  estimate $c_2^\prime$ by the larger bound: 
\begin{equation*}
    c_2^\prime \leq (4m^2n^2s)^{7m^2n^2s}.
\end{equation*}
Also,
\begin{equation*}
    P_T \leq Q_T \leq Q_S^{[M:K]}\leq Q_S^{m^2n^2}.
\end{equation*}
Let $C$ be the upper bound for $|D_M|$ from (i) of Lemma \ref{L_field_discriminant_bound}. Then, by Lemma~\ref{L_hR}, 
we have 
$$
h_M \leq 5C^{1/2}(\log C)^{m^2n^2d-1}, \qquad R_M\leq C^{1/2}(\log C)^{m^2n^2d-1}.
$$
Moreover, $A$ can be estimated from above using Lemma \ref{L_crucial}, and $B$ by (using Lemma~\ref{L_bound_roots_by_pol})
$$
\h(\alpha_1)+\h(\alpha_2)+ \log 2 \leq \h(f) +(n+1)\log 2.
$$
This implies that
\begin{equation}\label{eq:estimate_1}
    \begin{split}
        R_M&+h_Md_M^{-1}\log Q_T+md_MA+B\\ & \leq 7C^{1/2}(\log C)^{m^2n^2d-1} \log^* Q_S \\ &\leq 7C^{1/2}(\log C)^{m^2n^2d}.
    \end{split}
\end{equation}
Next, by (3.9) of \cite{BEG}, Lemma \ref{L_hR}, \eqref{eq:estimate_1} and the inequality $d + s^\prime \le 2s$, we have
\begin{equation}  \label{eq:R_T}
\begin{split}
    R_T  & \leq |D_M|^{1/2} (\log^*|D_M|)^{d_M -1} (\log^* P_T)^{t^\prime} \\
    &\leq C^{1/2}(\log C)^{m^2n^2d-1}(\log^* P_T)^{t^\prime}\\
    &\leq C^{1/2}(\log C)^{m^2n^2d-1}(m^2n^2\log^* Q_S)^{m^2n^2 s^\prime}\\
    &\leq (m^2n^2)^{m^2n^2s}C^{1/2}(\log C)^{2m^2n^2s-1}
\end{split}
\end{equation}
and
$$
1+\dfrac{\log^*R_T}{\log^*P_T}\leq 1 + \log^*R_T \leq 4m^2n^2s\log C,
$$
and so 
\begin{equation*}
    P_TR_T \left(1+\dfrac{\log^*R_T}{\log^*P_T}  \right)\leq (4m^2n^2)^{m^2n^2s}Q_S^{m^2n^2}C^{1/2}(\log C)^{2m^2n^2s}.
\end{equation*}
Combining all the above estimates yields
\begin{equation*}
    \begin{split}
        \h(\xi_1)&\leq 7m^6c_2^\prime(4m^2n^2)^{m^2n^2s}Q_S^{m^2n^2}C(\log C)^{4m^2n^2s}\\
        &\leq (4m^2n^2s)^{13m^2n^2s}Q_S^{m^2n^2}C^2.
    \end{split}
\end{equation*}
Since (using Lemma~\ref{L_bound_roots_by_pol})
\begin{equation*}
\begin{split}
\h(x) & \leq m\h(\xi_1) + \h(\gamma_1) + \h(\alpha_1)  +  \log2  \\
& \le m\h(\xi_1) + \h(\gamma_1) + \h(f)  +  (n+1)\log2, 
\end{split}
\end{equation*}
from (ii) of Lemma~\ref{L_crucial} and \eqref{def_ga_i} we can get an upper bound for $\h(\gamma_1)$ to obtain
\begin{equation} \label{eq:h(x)-0}
    \h(x) \leq 2m(4m^2n^2s)^{13m^2n^2s}Q_S^{m^2n^2}C^2.
\end{equation}
Finally, substituting $C$ (that is, the upper bound for $D_M$ in (i) of Lemma \ref{L_field_discriminant_bound}) 
gives the desired upper bound for $\h(x)$.
\end{proof}

\section{Proof of Theorem \ref{T_LeVeque}}

In the following proof we shall use some ideas from Trelina \cite{Trelina78},
Brindza \cite{B6}, Voutier \cite{Voutier2}, Bugeaud \cite{Bug2} and B\'erczes,
Evertse and Gy\H ory \cite{BEG}.

\begin{proof}[Proof of (i) of Theorem \ref{T_LeVeque}]
Let $x,y\in O_S$ be a solution of equation \eqref{eq:super} with $y\neq 0$,
and recall that in this case we are working under the assumption
$m_1=m_2=m_3=2$, which implies $r \ge 3$. Then by Lemma \ref{L_crucial} we have
\begin{equation}\label{eq:crucial_222}
x-\al_i=\ga_i\xi_i^2, \quad\quad \text{where} \quad
\ga_i=a_0^{-1}(\de_{i1}/\de_{i2})\in L_i, \ \xi_i\in O_{T_i}
\end{equation}
for $i=1,2,3$. Let now $E$ denote the upper bound for $\h(\de_{ij})$ given
in Lemma \ref{L_crucial} with $m_i=2$, $i=1,2,3$, $j=1,2$. 

Let
$$
M=K(\al_1,\al_2,\al_3,\sqrt{\ga_1/\ga_2},\sqrt{\ga_1/\ga_3}),
$$
$d_M=[M:\Q]$, $T$ the set of places of $M$ lying above the places from $S$,
$O_T$ the ring of $T$-integers in $M$, and $t$ the cardinality of $T$.
Then, we have $d_M\leq 4dr^3 \leq 8r^3 s$ and $t\leq 4r^3s$, where $s$ is the cardinality of $S$.

It follows from \eqref{eq:crucial_222} that
\begin{equation}\label{eq:Pell}
\ga_1\xi_1^2-\ga_2\xi_2^2=\al_2-\al_1, \quad\quad \ga_1\xi_1^2-\ga_3\xi_3^2=\al_3-\al_1
\quad \text{with} \ \xi_1,\xi_2,\xi_3 \in O_T.
\end{equation}
We now apply Proposition \ref{P_Pell} to the simultaneous Pellian equations \eqref{eq:Pell}. Denote by $h_M, R_M$ the class number and the regulator of $M$, and
by $P_T, R_T, Q_T$ the quantities in $M$ corresponding to $P_S, R_S, Q_S$, respectively. Let $\h(\ga_i)\leq A$ for $i=1,2,3$ and $B=\max(\h(\al_1-\al_2), \h(\al_1-\al_3))$.
Then Proposition \ref{P_Pell} gives
\begin{equation}\label{eq:xi_i_estimate}
\max_{1\leq i \leq 3} \h(\xi_i)\leq cP_TR_T\left( 1+\frac{\log^* R_T}{\log^* P_T}\right)
\left( R_M + \frac{h_M}{d_M}\log Q_T+d_M A+B \right),
\end{equation}
where we can take $c = t^{2t+4} 2^{7t+60} d_M^{2t+d_M +2}$.

We now estimate from above the parameters occurring in \eqref{eq:xi_i_estimate}.
We have $A\leq 2E+\h(a_0)\leq 2E + H(f^*)$, 
$$
c \leq (16r^3 s)^{25r^3 s}, \quad\quad P_T \le Q_T\leq (Q_S)^{4r^3}, 
$$
and
$$
h_M \leq 5C^{1/2}(\log C)^{4dr^3-1}, \qquad  R_M \leq C^{1/2}(\log C)^{4dr^3-1}, 
$$
where $C$ denotes the upper bound on $|D_M|$ from (ii) of Lemma \ref{L_field_discriminant_bound}.
Hence, as in \eqref{eq:estimate_1} and \eqref{eq:R_T}, we obtain 
$$
R_M+h_M d_M^{-1} \log Q_T+ d_MA + B\leq 7C^{1/2}(\log C)^{4dr^3}, 
$$
and 
$$
R_T \leq (4r^3)^{4r^3 s}C^{1/2}(\log C)^{8r^3 s -1}, 
$$
and then 
$$
1+\dfrac{\log^*R_T}{\log^*P_T}\leq 16 r^3 s\log C. 
$$
Therefore, by \eqref{eq:xi_i_estimate} and using at some point $d\leq 2s$, after some straightforward computation we get
\begin{equation*}
\begin{split}
\max_{i=1,2,3} \h(\xi_i)  \leq (16r^3 s)^{47r^3 s} (Q_S)^{4r^3} C^2.
\end{split}
\end{equation*}

Now it follows from \eqref{eq:crucial_222} and Lemma~\ref{L_bound_roots_by_pol}  that
\begin{equation*}
\begin{split}
\h(x) & \leq  2\h(\xi_1)+ \h(\ga_1) + \h(\al_1) + \log 2 \\
& \leq 2\h(\xi_1)+ \h(\ga_1) + \h(f^*) + (r+1)\log 2, 
\end{split}
\end{equation*}
and so (as in \eqref{eq:h(x)-0}), 
$$
\h(x) \leq 4 \cdot (16r^3 s)^{47r^3 s} (Q_S)^{4r^3} C^2. 
$$
Finally, by substituting $C$ and noticing $r \ge 3$, we obtain  \eqref{eq:bound222}, as claimed.
\end{proof}


\begin{proof}[Proof of (ii) of Theorem \ref{T_LeVeque}]
Let $m\geq 3$, and suppose that there are $i,j$ with $1\leq i \ne j<r$ such that
$m'= \gcd(m_i,m_j)\geq 3$. We  write $m_i=m'm_i'$, $m_j=m'm_j'$, where $m_i', m_j'$
are positive integers.

Let $x,y\in O_S$ be a solution of equation \eqref{eq:super} with $y\ne 0$.
Then in view of Lemma \ref{L_crucial} we obtain
\begin{equation}\label{eq:crucial_m'}
x-\al_i=\ga_i\left(\xi_i^{m_i'}\right)^{m'}, \quad\quad x-\al_j=\ga_j\left(\xi_j^{m_j'}\right)^{m'}
\end{equation}
with $\ga_i,\ga_j$ and $\xi_i,\xi_j$ as in Lemma \ref{L_crucial} and  \eqref{def_ga_i}.

We now follow the proof of Proposition \ref{T_super_simple}.
Let 
$$
M=K(\al_i,\al_j,\sqrt[m']{\ga_i/\ga_j},\rho),
$$
 where $\rho$ is
a primitive $m'$-th root of unity. Let $T$ be the set of places
of $M$ lying above the places from $S$. Then \eqref{eq:crucial_m'} implies that
\begin{equation}\label{eq:resulting_thue}
\ga_i\left(\xi_i^{m_i'}\right)^{m'}-\ga_j\left(\xi_j^{m_j'}\right)^{m'}=\al_j - \al_i
\quad\quad \text{in} \quad\quad \xi_i^{m_i'}, \xi_j^{m_j'} \in O_T,
\end{equation}
where the left hand side can be regarded as a binary form of non-zero
discriminant which splits into linear factors over $M$. We can now apply
Proposition \ref{P_Thue} to the equation \eqref{eq:resulting_thue} to get the upper bound
\begin{equation}\label{bound_interm_13}
c_2''m^6P_TR_T\left( 1+ \frac{\log^* R_T}{\log^* P_T}\right)
(R_M+h_Md_M^{-1}\log Q_T+md_MA+B),
\end{equation}
for $\h\left(\xi_i^{m_i'}\right), \h\left(\xi_j^{m_j'}\right)$ and hence for
$\h\left(\xi_i\right), \h\left(\xi_j\right)$, with the notation
$P_T, R_T, R_M, h_M$ and $Q_T$ from the proof of Proposition \ref{T_super_simple}.
Further, $A=\max(\h(\ga_i),\h(\ga_j))$, $B=\h(\al_i-\al_j)$, $d_M=[M:\Q]$
and $c_2''$ is the constant $c_2$ from Proposition \ref{P_Thue}
but with $s,d$ replaced by the upper bounds $m^2r^2s$ and $m^2r^2d$ 
for the cardinality of $T$ and $[M: \Q]$, respectively. Then, bounding the parameters
$P_T, R_T, \dots$ in \eqref{bound_interm_13} in the same way as in the proof of Proposition \ref{T_super_simple}, we get \eqref{eq:bound_m'} as claimed. Essentially, we obtain \eqref{eq:bound_m'} by replacing $n$ with $r$ in Proposition~\ref{T_super_simple}. 
\end{proof}


\begin{proof}[Proof of (iii) of Theorem \ref{T_LeVeque}]
Let $m\geq 3$, and suppose that the hypotheses in (i) and (ii) of Theorem \ref{T_LeVeque} are not fulfilled. Then there exist $1\leq i\ne j\leq r$ such that $m_i\geq 3$ and
$m_j\geq 2$.

Let $x,y\in O_S$ be a solution of \eqref{eq:super}. Then by Lemma \ref{L_crucial} and \eqref{def_ga_i}, we have
\begin{equation}\label{eq:crucial_32}
a_0(x-\al_i)=\ga_i'\xi_i^{m_i}, \quad\quad a_0(x-\al_j)=\ga_j'\xi_j^{m_j}
\end{equation}
where
$\ga_i'=a_0\ga_i \in K(\al_i)$, $\ga_j'=a_0\ga_j \in K(\al_j)$, $\xi_i\in O_{T_i}$,
$\xi_j\in O_{T_j}$.

Next we use some ideas from Voutier \cite{Voutier2}. If $\al_i$ is of degree at least $2$ over $K$, 
then we can take a conjugate, say $\al_{i'}$ of $\al_i$ over $K$ which is also a
root of $f(X)$. Further it follows from \eqref{eq:crucial_32} that $m_{i'}=m_i$.
So we arrive at such a case for which (ii) of Theorem \ref{T_LeVeque} applies.
However, this is excluded, thus $\al_i\in K$. Similarly, if $\al_j$
is of degree $\geq 3$ over $K$ we get the cases (i) or (ii) of Theorem \ref{T_LeVeque},
which are also excluded. Hence, $\al_j$ is of degree at most $2$ over $K$.

We note that $a_0(x-\al_i) \in O_{T_i}$ and $a_0(x-\al_j) \in O_{T_j}$. Let
$$
M=K(\al_i, \al_j,\sqrt[m_i]{\ga_i'},\sqrt[m_j]{\ga_j'},\rho),
$$
where $\rho$ is a primitive $m_im_j$-th root of unity. Denote by $T$ the set of
places of $M$ lying above those of $S$, and by $O_T$ the ring of $T$-integers
in $M$. Then $O_{T_i}, O_{T_j}$ are contained in $O_T$, and \eqref{eq:crucial_32}
implies that
$$
\tau_i:=\sqrt[m_i]{\ga_i'}\xi_i, \quad \tau_j:=\sqrt[m_j]{\ga_j'}\xi_j
$$
are in $O_T$. Further, by \eqref{eq:crucial_32} we know that $\widetilde{x}=\tau_j$, $\widetilde{y}=\tau_i$ is a solution of the equation
\begin{equation}\label{eq:super_32}
g(\widetilde{x})=\widetilde{y}^{m_i} \quad\quad \text{in} \quad \widetilde{x},\widetilde{y}\in O_T,
\end{equation}
where the polynomial $g(X)=X^{m_j}+a_0(\al_j - \al_i)$.

We now apply Proposition \ref{T_super_simple} to the equation \eqref{eq:super_32} to obtain a bound for
the height of $\tau_j$. Then, putting $t=|T|$ we get
\begin{equation}\label{eq:bound_tau_j}
\h(\tau_j) \leq (6m_jt)^{14m_i^3m_j^3t}|D_M|^{2m_i^2m_j^2}H(g)^{8m_i^2m_j^3 d_M}Q_T^{3m_i^2m_j^2},
\end{equation}
where $d_M = [M:\Q]$ and $Q_T$ denotes the product of norms of prime ideals in $T$.

We now estimate from above the parameters in \eqref{eq:bound_tau_j} in terms of
$d,m,r, s,|D_K|,Q_S,H(f^*)$ and $N_S(b)$.

We have $m_i \le m$, $m_j\leq m$, $d_M \leq 2dm_i^2m_j^2 \leq 2dm^4$,
$t\leq [M:K]s\leq 2m^4s$, and $Q_T\leq Q_S^{[M:K]}\leq Q_S^{2m^4}$. Further, 
$\h(g)=\h(a_0(\al_i-\al_j))$ and by Lemma~\ref{L_bound_roots_by_pol} we have 
\begin{align*}
\h(a_0(\al_i-\al_j)) & \leq \h(a_0) + \h(\al_i-\al_j) \\
& \leq \h(a_0) + \h(\al_i) + \h(\al_j) + \log 2 \\
& \leq 2\h(f^*) + (r+1)\log2,
\end{align*}
and so $H(g)\leq 2^{r+1}H(f^*)^{2}$. After some computations we deduce from
\eqref{eq:bound_tau_j} that
$$
\h(\tau_j)\leq 2^{16(r+1)dm^9}(12m^5s)^{28m^{10}s}|D_M|^{2m^4}H(f^*)^{32dm^9}Q_S^{6m^8}=:C.
$$
It follows from \eqref{eq:crucial_32} that
$$
x=\al_j+a_0^{-1}\tau_j^{m_j},
$$
whence we get that
\begin{align*}
\h(x) & \leq \h(\al_j)+\h(a_0)+m_j\h(\tau_j)+\log 2 \\
& \leq m\h(\tau_j) + 2\h(f^*) + (r+1)\log 2 \\
& \leq (m+1)C.
\end{align*}
Finally, noticing $m \ge 3, r\ge 2$ and estimating $|D_M|$ from above by the upper bound occurring in (iii) of Lemma \ref{L_field_discriminant_bound}, after some straightforward computations we obtain 
\eqref{eq:bound32}, as claimed.
\end{proof}


\section{Preparations for proving Theorem~\ref{thm:GenS-T}}

To prove Theorem~\ref{thm:GenS-T} we follow the proof of Theorem 2.3 of \cite{BEG} and modify it according to the more
general conditions. We first make some preparations as in  Section 5 of \cite{BEG}. 
For the sake of completeness we repeat estimates which are identical to the original ones.


Recall that in some finite extension of $K$, the polynomial
$f \in O_S[X]$ in Theorem~\ref{thm:GenS-T} factorizes as
$$
f(X) = a_0\cdot \prod_{i=1}^r (X-\al_i)^{e_i}
$$
with distinct $\al_1, \dots, \al_r$. Let us recall again some previous notation.
For $i=1\dots r$, let $L_i=K (\alpha_i)$, and by $d_{L_i},h_{L_i}, D_{L_i}, R_{L_i}$ we denote
the degree, class number, discriminant and regulator of $L_i$. Further, let $T_i$ be the set of places of $L_i$
lying above the places in $S$ and let $R_{T_i}$ denote the $T_i$-regulator
of $L_i$, and $t_i$ the cardinality of $T_i$. Let $O_{T_i}$ denote the ring of
$T_i$-integers of $L_i$, and by $O_{T_i}^*$ its unit group.
Put
$$
Q_{T_i}=\prod_{\fP\in T_i} N_{L_i}(\fP),
$$
 where the product is over all the prime ideals in $T_i$
 (if there is no prime ideal in $T_i$, then $Q_{T_i}$ is set to be $1$).

\begin{lemma}  \label{lem:T-units}
The group of $T_i$-units $O_{T_i}^*$ is finitely generated, and we may choose a fundamental system of $T_i$-units (i.e., a basis of $O_{T_i}^*$ modulo torsion) $\eta_{i1}, \dots, \eta_{i,t_i-1}$ with the property
\begin{equation*}
\left\{\begin{array}{l}
\displaystyle{\prod_{j=1}^{t_i-1} h(\eta_{ij})\leq c_{1i}R_{T_i}},
\\
\displaystyle{\max_{1\leq j\leq t_i-1} h(\eta_{ij}) \leq c_{2i}R_{T_i}},
\end{array}\right.
\end{equation*}
where
\begin{equation*}
c_{1i},c_{2i}\leq 1200t_i^{2t_i}\leq 1200 (ns)^{2ns}.
\end{equation*}
\end{lemma}

\begin{proof}
This is in fact Lemma 2 of \cite{GY} (see also \cite{BugGy1}, \cite{BugGy2} and \cite{Bug1}) and the estimates follow form Lemma 2 of \cite{GY} by
using $t_i\leq [L_i:K]s\leq ns$.
\end{proof}

For the class number and regulator $h_{L_i}$, $R_{L_i}$,
by Lemma~\ref{L_hR} we have 
\begin{equation*}
\max (h_{L_i},R_{L_i})
\leq 5|D_{L_i}|^{1/2}(\log^* |D_{L_i}|)^{d_{L_i}-1}, 
\end{equation*}
and then as in \eqref{eq:hR_estimate}, using $d_{L_i}\leq nd$, $d \le 2s$ and Lemma~\ref{lem:discr_of_ext}  we obtain 
\begin{equation} \label{h_LR_L-bound}
\begin{split}
&\max(h_{L_i},R_{L_i})\leq (nd)^{nd} |D_{L_i}| \le 4^{(n-1)nd} (n^3 d)^{nd} e^{(2n-2)d\h(f)} |D_{K}|^n   \\
& \qquad\qquad\qquad\qquad\qquad\qquad\,\,  \le (4^n n^3 s)^{2ns} e^{(2n-2)d\h(f)} |D_{K}|^n := C_0.
\end{split}
\end{equation}
For the $T_i$-regulator $R_{T_i}$, by (3.9) of \cite{BEG} we have 
\begin{equation*}
R_{T_i}\leq |D_{L_i}|^{\frac{1}{2}}(\log^* |D_{L_i}|)^{d_{L_i}-1}\cdot (\log^* P_{T_i})^{t_i -1}, 
\end{equation*}
and then as in \eqref{eq:hR_estimate} and \eqref{h_LR_L-bound}, using also $t_i\leq ns$ we get 
\begin{equation} \label{R_S_bound}
\begin{split}
&R_{T_i}\leq (nd)^{nd} |D_{L_i}| (\log^* P_{T_i})^{ns-1}  \le C_0(n\log^* P_S)^{ns-1}  \\
& \qquad\qquad\qquad\qquad\qquad  \le (4^{2n}n^7 s^2)^{ns} e^{(2n-2)d\h(f)} |D_{K}|^n (\log^* P_S)^{ns-1}.
\end{split}
\end{equation}
So, combining Lemma~\ref{lem:T-units} with \eqref{R_S_bound}, we have 
\begin{equation}\label{eq:C1}
\begin{aligned}
&\prod_{j=1}^{t_i -1} \h(\eta_{ij}) \leq 1200  (4^{2n}n^9 s^4)^{ns} e^{(2n-2)d\h(f)}
|D_{K}|^n  (\log^* P_S)^{ns-1} :=C_1,\\
&\max_{1\leq j\leq t_i -1} \h(\eta_{ij}) \leq C_1.
\end{aligned}
\end{equation}

The below lemma is a straightforward generalization of Lemma 5.1 of \cite{BEG} and it is crucial in the proof of Theorem~\ref{thm:GenS-T}.

\begin{lemma}\label{lem:fact}
Let $x,y,m$ satisfy the equation \eqref{eq:super} with $y\ne 0$ and $y \not\in O_S^*$.
Then, for $i=1,2$ there are $\gamma_i, \xi_i\in L_i$ and integers $b_{i1}, \dots ,b_{i,t_i-1}$
of absolute value at most $\frac{m}{2}$ such that
$$
\left\{\begin{array}{l}
(x-\al_i)^{e_1e_2h_{L_1}h_{L_2}}=\eta_{i1}^{b_{i1}}\cdots  \eta_{i,t_i-1}^{b_{i,t_i-1}}\gamma_i \xi_i^m,
\\[0.2cm]
\h(\gamma_i)\leq C_2:= (4^n n^4 s)^{4ns} H(f)^{4dn} |D_K|^{2n}  (\log^* P_S) (\log^* N_{S} (b)).
\end{array}\right.
$$
\end{lemma}

\begin{proof}
The proof follows the proof of Lemma 5.1 of \cite{BEG}, however for convenience
we do not refer back to that paper, but instead, give a complete proof of our generalization, thus including estimates which already appeared in \cite{BEG}.

By symmetry, it suffices to prove the result for $i=1$. 
Put $q=e_2h_{L_1}h_{L_2}$, and as in Lemma \ref{L_crucial} define again
$$
g_1(X) = \frac{f(X)}{a_0 (X-\al_1)^{e_1}} \ \ \text{and} \ \ f^*(X) = a_0\cdot \prod_{i=1}^{r} (X-\al_i).
$$

Let $M$ denote the splitting field of $f$ over $K$.
By $[\cdot]$ we denote fractional ideals with respect to the integral closure of $O_{T_1}$ in $M$.
In addition, for any polynomial $g \in M[X]$, denote by $D(g)$ the discriminant of $g$
and  denote by $[g]$ the fractional ideal generated by the coefficients of $g$.

Then (as in Lemma \ref{L_crucial}) we have again the relations
\eqref{eq:factor_gcd} and \eqref{eq:super_ideal}, where now in \eqref{eq:super_ideal} we have $[y] \ne [1]$ because $y \not\in O_S^*$.
Note that the generators of the ideals occurring in \eqref{eq:factor_gcd} and \eqref{eq:super_ideal} are all defined over $L_1$, so we can view them as fractional ideals of  $O_{T_1}$.
In the sequel, without confusion we use $[\cdot]$ to denote fractional ideals of  $O_{T_1}$.

From these relations we see that there are integral ideals
$\fB_1,\fB_2$ of $O_{T_1}$ and a fractional ideal $\fA$ of $O_{T_1}$, such that
\[
\frac{[x-\alpha_1 ]^{e_1}}{[1,\alpha_1 ]^{e_1}}=\fB_1\fB_2^{-1}\fA^m,
\]
where
\[
\fB_1\supseteq [b] \cdot \frac{[D(f^*)]^{e_1(n-e_1)}}{[f^*]^{2n-2}},\ \
\fB_2\supseteq [f]\cdot \frac{[D(f^*)]^{e_1(n-e_1)}}{[f^*]^{2n-2}}.
\]
Here, since $D(f^*) \in [f^*]^{2r - 2}$ (due to the determinantal expression of $D(f^*)$),
$[D(f^*)] / [f^*]^{2r - 2}$ is in fact an integral ideal of $O_{T_1}$, and so is  $[D(f^*)]^{e_1(n-e_1)} / [f^*]^{2n-2}$.

Since
\[
[a_0][1,\alpha_1 ]^{e_1}\subseteq [a_0]\prod_{i=1}^{r} [1,\alpha_i]^{e_i}
\subseteq [f]\subseteq [1],
\]
it follows that $[1,\alpha_1 ]^{-e_1}\supseteq [a_0]$,
thus
\[
[x-\alpha_1 ]^{e_1}=\fC_1\fC_2^{-1}\fA^m,
\]
where $\fC_1,\fC_2$ are ideals of $O_{T_1}$ such that
\[
\fC_1,\fC_2\supseteq [a_0 b D(f^*)^{e_1(n-e_1)}].
\]
By raising to the power $q$ (note that $q=e_2h_{L_1}h_{L_2}$), we get
\begin{equation}\label{eq:x-alpha}
(x-\alpha_1 )^{e_1q} = (x-\alpha_1 )^{e_1e_2 h_{L_1} h_{L_2}} = \beta_1\beta_2^{-1}\lambda^m,
\end{equation}
for some non-zero $\beta_1,\beta_2\in O_{T_1}$ and $\lambda\in L_1^*$ with
\[
[\beta_k]\supseteq [a_0 b D(f^*)^{e_1(n-e_1)}]^q\ \ \mbox{for } k=1,2.
\]
By Lemma \ref{L_bound_hight_by_norm}
there exist $\varepsilon_1,\varepsilon_2\in O_{T_1}^*$
such that for $k=1,2$,
\[
\h(\varepsilon_k\beta_k)\leq \frac{r}{d_{L_1}}\log N_{T_1}(a_0 b D(f^*)^{e_1(n-e_1)})
+cR_{L_1}+\frac{h_{L_1}}{d_{L_1}}\log Q_{T_1},
\]
where the norm function $N_{T_1}$ is defined as in \eqref{eq:S-norm},
and $c = 39d_{L_1}^{d_{L_1}+2}\leq 39(2ns)^{2ns+2}$.

Now, since $\varepsilon_2\cdot \varepsilon_1^{-1}\in O_{T_1}^*$, we have
\begin{equation}
\label{eq:eps12}
\varepsilon_2\varepsilon_1^{-1}
=\zeta \varepsilon^m\eta_{1,1}^{b_{1,1}}\cdots\eta_{1,t_1-1}^{b_{1,t_1-1}}.
\end{equation}
with a suitable root of unity $\zeta$ of $L_1$,
and integers $b_{1,1}\dots b_{1,t_1-1}$ of absolute value at most $m/2$, and with a suitable $\varepsilon\in O_{T_1}^*$.
Here  $\eta_{1,1},\dots ,\eta_{1,t_1-1}$ are the fundamental units of $O_{T_1}^*$ satisfying
the last two inequalities in \eqref{eq:C1}.

Writing
$$
\gamma = \zeta \frac{\varepsilon_1\beta_1}{\varepsilon_2\beta_2}, \qquad
\xi =\varepsilon\lambda,
$$
we get from \eqref{eq:x-alpha} and~\eqref{eq:eps12} that
\[
(x-\alpha_1)^{e_1q} = \eta_1^{b_{1,1}}\cdots\eta_{t_1 -1}^{b_{1, t_1 -1}}\gamma\xi^m,
\]
 and
\begin{equation}\label{eq:gamma-bound}
\h(\gamma )\leq \frac{2q}{d_{L_1}}\log N_{T_1}(a_0 b D(f^*)^{e_1(n-e_1)})
+2cR_{L_1}+2\frac{h_{L_1}}{d_{L_1}}\log Q_{T_1}.
\end{equation}

Now using Lemma \ref{disc-height} and Lemma~\ref{lem:f*}, we obtain
\begin{equation}  \label{eq:NT}
\begin{split}
 d_{L_1}^{-1}\log & N_{T_1}(a_0 b D(f^*)^{e_1(n-e_1)}) \\
 & \le \h( a_0 D(f^*)^{e_1(n-e_1)} ) + \log N_{S} (b) \\
& \le \h(a_0)  + e_1(n-e_1) \h(D(f^*)) + \log N_{S} (b)  \\
& \le \h(f) + \frac{1}{4} n^2 \big( (2n-1)\log n + (2n-2)\h(f^*) \big) + \log N_{S} (b)  \\
& \le \h(f)   + n^2(n-1)\h(f) + n^4 + \log N_{S} (b) \\
& \le n^3 \h(f) + n^4 + \log N_{S} (b) .
\end{split}
\end{equation}

We clearly have
$$
d_{L_1}^{-1} \log Q_{T_1} \le d^{-1} \log Q_S \le s \log^* P_S.
$$
Hence,  applying the above estimates, the estimates in \eqref{eq:C1}, $d\leq 2s$ and $n \ge 2$ to \eqref{eq:gamma-bound},
after a straightforward computation we get
\begin{align*}
\h(\gamma)  & \le 2(n-1)C_0^2 \Big( n^3 \h(f) + n^4 + \log N_{S} (b)\Big) \\
& \qquad\qquad\qquad + 78(2ns)^{2ns+2}C_0 + 2C_0 s \log^* P_S \\
& \le C_0^2 \cdot n^{4ns}e^{4d\h(f)} (\log^* P_S) (\log^* N_{S} (b)), \\
& = (4^n n^4 s)^{4ns} H(f)^{4dn} |D_K|^{2n} (\log^* P_S) (\log^* N_{S} (b)),
\end{align*}
as claimed in the lemma.
\end{proof}

\section{Proof of Theorem~\ref{thm:GenS-T}}

To prove our Theorem~\ref{thm:GenS-T} we follow the proof of Theorem 2.3 of \cite{BEG}.

Let $L=K(\alpha_1,\alpha_2)$, $d_L =[L:\Q ]$,
$T$ the set of places of $L$ lying above the places from $S$, and
$t$ the cardinality of $T$. Clearly, we have
$$
t \le n(n-1)s \le n^2 s.
$$
Let again $x,y\in O_S$ and $m\geq 3$ an integer
satisfying the equation \eqref{eq:super} and such that $y$ is non-zero and not  in $O_S^*$.

Put
$$
W =\max_{1\leq i \leq r} \h(x-\al_i).
$$
Similarly as in \cite{BEG}, we assume  without loss of generality that
\begin{equation}  \label{eq:m low1}
m \ge 4^{4n^2 s}(10n^2 s)^{38ns} e^{11nd \h(f)} |D_K|^{6n} P_S^{n^2}  \log^* N_S(b).
\end{equation}
Then, we have
\begin{equation}  \label{eq:W low1}
W \ge \max \left(C_3, \frac{m \log 2}{2nd} \right),
\end{equation}
where
$$
C_3 = 4^{4n^2 s}(10n^2 s)^{37ns} e^{11nd \h(f)} |D_K|^{6n} P_S^{n^2}  \log^* N_S(b).
$$
Indeed, we have assumed that $b \in O_S \backslash \{0\}$, $y \ne 0$ and $y \in O_S \backslash O_S^*$, which gives $N_S (b) \ge 1$ and  $N_S (y) \ge 2$, thus
$$
2^m \le N_S(y^m) \le N_S(by^m)  \leq e^{d\h(by^m)} =e^{d\h(f(x))} \leq e^{dn W + d\h(a_0)},
$$
which gives
\begin{equation}   \label{eq:m up1}
m \leq \frac{dnW + d\h(f)}{\log 2}.
\end{equation}
Suppose $W < C_3$,  from \eqref{eq:m up1} we can obtain a upper bound for $m$ which contradicts \eqref{eq:m low1}.
So, we must have $W \ge C_3$.
In addition, suppose $W < \frac{m \log 2}{2nd}$, that is  $m > \frac{2dnW}{\log 2}$,
which together with \eqref{eq:m up1} gives $W < \h(f)/n$, and this contradicts  $W \ge C_3$.
Hence, we obtain \eqref{eq:W low1}.


Without loss of generality, we suppose that
$$
W = \h(x-\al_2). 
$$
If $|x-\al_2|_v\leq 1$ for every $v\in T$, then Letting $M_L$ be the set of places of $L$ and using $x\in O_S$ and Lemma~\ref{L_bound_roots_by_pol}, we have
\begin{align*}
W&=\frac{1}{d_L}\log \left( \prod_{v \in M_L} \max(1,|x-\al_2|_v)\right) =\frac{1}{d_L}\log \left( \prod_{v \not\in T} \max(1,|x-\al_2|_v)\right) \\
&\leq \frac{1}{d_L}\log \left( \prod_{v \not\in T} \max(1,|\al_2|_v)\right) \leq \h(\al_2) \leq \h(f) + n\log 2,
\end{align*}
which is impossible by \eqref{eq:W low1}.
Consequently we must have $\max_{v\in T} |x-\alpha_2|_v>1$, 
and we can choose a valuation $v_0 \in T$ such that
\begin{equation}\label{max_val}
|x-\al_2|_{v_0}=\max_{v\in T} |x-\al_2|_v > 1.
\end{equation}
Then, we have
$$
\begin{aligned}
W&\leq \frac{1}{d_L}\left( \log \left( |x-\al_2|^t_{v_0}\prod_{v \not\in T} \max(1,|x-\al_2|_v)\right) \right)\\
&\leq \frac{1}{d_L}\left( \log \left( |x-\al_2|^t_{v_0}\prod_{v \not\in T} \max(1,|\al_2|_v)\right) \right),
\end{aligned}
$$
which gives
$$
|x-\al_2|_{v_0} \geq
\frac{e^{Wd_L/t}}{\prod_{v \not\in T} \max(1,|\al_2|_v)^{1/t}}.
$$
Thus we have
\begin{equation}\label{bound_interm_11a}
\begin{aligned}
\left| 1- \frac{x-\al_1}{x-\al_2} \right|_{v_0} &=
\frac{|\al_2-\al_1|_{v_0}}{|x-\al_2|_{v_0}}
&\leq  \frac{|\al_2-\al_1|_{v_0}\prod_{v \not\in T} \max(1,|\al_2|_v)^{1/t}}{e^{Wd_L/t}}.
\end{aligned}
\end{equation}
Put $\delta(v_0)=1$ if $v_0$ is real, $\delta(v_0)=2$ if $v_0$ is complex,
and $\delta(v_0)=0$ if $v_0$ is finite.
Since by Lemma \ref{L_bound_roots_by_pol} we have
$$
\begin{aligned}
|\al_2-\al_1|_{v_0}&\prod_{v \not\in T} \max(1,|\al_2|_v)^{1/t}
\\
&\leq 2^{\delta(v_0)}\max(1,|\al_2|_{v_0})\max(1,|\al_1|_{v_0})\prod_{v \not\in T} \max(1,|\al_2|_v)
\\
&\leq 2^{\delta(v_0)}\exp(d_L(\h(\al_1)+\h(\al_2)))
\\
&\leq 2^{\delta(v_0)+ nd_L}\exp(d_L \h(f)),
\end{aligned}
$$
thus the estimate \eqref{bound_interm_11a} gives us
\begin{equation}\label{bound_interm_11}
\left| 1- \frac{x-\al_1}{x-\al_2} \right|_{v_0}\leq
\exp\Big( (\delta(v_0)+nd_L)\log 2\, +d_L \h(f)-Wd_L/t\Big).
\end{equation}
Notice that by \eqref{eq:W low1} and \eqref{bound_interm_11}, we have (the inequality can be proved by contradiction)
\begin{equation}\label{bound_interm_12}
\left| 1- \frac{x-\al_1}{x-\al_2} \right|_{v_0} <1.
\end{equation}
Since for any $y\in L$ with $|1-y|_{v_0}<1$ and any
positive integer $k$ we have
\[
|1-y^k|_{v_0}\leq 2^{k\cdot \delta(v_0)}|1-y|_{v_0},
\]
thus we obtain
\begin{align*}
&  \left| 1-  \left(\frac{x-\al_1}{x-\al_2}\right)^{e_1 e_2 h_{L_1}h_{L_2}} \right|_{v_0} \\
& \quad \leq
\exp\Big( (e_1e_2h_{L_1}h_{L_2}\delta(v_0)+ \delta(v_0) + nd_L)\log 2 + d_L \h(f)-Wd_L/t\Big).
\end{align*}
Now using again \eqref{eq:W low1} and the estimate
\eqref{h_LR_L-bound}, $d_L\leq n^2d$ and $s\leq t\leq n^2s$,
this can be simplified to
\begin{equation}  \label{eq:upper}
\left\vert 1-\left( \frac{x-\al_1}{x-\al_2}\right)^{e_1e_2h_{L_1}h_{L_2}} \right\vert_{v_0}
\leq \exp \left( -Wd_L/(2t) \right).
\end{equation}

Now, we want to get a lower bound for the left hand side of \eqref{eq:upper} by using Baker's method
and then complete the proof.
But before this we need to handle the exceptional case when
$$
1-\left( \frac{x-\al_1}{x-\al_2}\right)^{e_1e_2h_{L_1}h_{L_2}} = 0.
$$
This implies that $(x- \alpha_1) / (x-\alpha_2)$ is a root of unity, say $\zeta$.
Since $\alpha_1$ and $\alpha_2$ are distinct, we have $\zeta \ne 1$.
Then, we have
$$
x = \frac{\alpha_1 - \zeta \alpha_2}{1-\zeta}.
$$
So, we get
$$
\h(x) \le \h(\alpha_1 - \zeta \alpha_2) + \h(1-\zeta)
\le \h(\alpha_1) + \h(\alpha_2) + \log 2 + \log 2,
$$
which, together with Lemma \ref{L_bound_roots_by_pol} gives
$$
\h(x) \le \h(f) + (n+2)\log 2.
$$
Hence, we obtain
\begin{align*}
W = \h(x - \alpha_2) & \le \h(x) + \h(\alpha_2) + \log 2 \\
& \le \h(f) + \h(\alpha_2) + (n+3)\log 2 \\
& \le 2\h(f) +  (2n+3)\log 2,
\end{align*}
where the last inequality again follows from Lemma \ref{L_bound_roots_by_pol}.
Then, in view of \eqref{eq:m up1}, we have
$$
m \le \frac{dn \big(2\h(f) +  (2n+3)\log 2 \big) + d\h(f)}{\log 2},
$$
which yields a much better upper bound for $m$ than the one claimed in Theorem~\ref{thm:GenS-T}.

So, in the sequel we assume that
$$
1- \left(\frac{x-\al_1}{x-\al_2}\right)^{e_1e_2h_{L_1}h_{L_2}} \ne 0.
$$
Then, using Lemma~\ref{lem:fact} and Proposition \ref{P_Baker} we get 
\begin{equation}\label{eq:lower}
\begin{aligned}
&\left| 1- \left(\frac{x-\al_1}{x-\al_2}\right)^{e_1e_2h_{L_1}h_{L_2}} \right|_{v_0}
\\
&\qquad =
\left| 1-
 \eta_{11}^{b_{11}}\cdots \eta_{1,t_1-1}^{b_{1,t_1-1}}
\cdot \eta_{21}^{-b_{21}}\cdots \eta_{2,t_2-1}^{-b_{2,t_2-1}}
\cdot  \frac{\gamma_1}{\gamma_2}
\cdot \left(\frac{\xi_1}{\xi_2}\right)^m  \right|_{v_0}
\\
&\qquad \geq \exp\Big(-c \cdot  \frac{N(v_0)}{\log N(v_0)}\Theta \log B \Big),
\end{aligned}
\end{equation}
where  (noticing $ \eta_{11},\ldots \eta_{1,t_1-1}, \eta_{21},\ldots \eta_{2,t_2-1}$ are all not roots of unity,
and so their heights are all not less than Voutier's lower bound $V(d)$ from
Lemma \ref{L_lower bound_for hight})
\begin{align*}
&\Theta=\max \big(\h(\frac{\gamma_1}{\gamma_2}),V(d) \big) \cdot \max \big(\h(\frac{\xi_1}{\xi_2}),V(d) \big)\cdot
\prod_{j=1}^{t_1-1}\h(\eta_{1j})\cdot
\prod_{j=1}^{t_2-1}\h(\eta_{2j}),
\\
&B=\max\big(3, m, |b_{11}|, \ldots, |b_{1,t_1-1}|, |b_{21}|, \ldots, |b_{2,t_2-1}|\big),
\\
&N(v_0)=\begin{cases} 2 & \text{if $v_0$ is infinite,}
\\
N_L(\fP) & \text{if $v_0=\fP$ is finite,}\end{cases} \\
& c = 12(16ed_L)^{3(t_1+t_2)+2} (\log^* d_L)^2. \\
\end{align*}

We now estimate the above parameters.
First, by Lemma~\ref{lem:fact}, we have that $\h(\gamma_i) \le C_2, i=1,2$,
and the exponents $b_{ij}$ in \eqref{eq:lower} have absolute values at most $m/2$.
So, we have
$$
\h\big(\frac{\gamma_1}{\gamma_2} \big) \le \h(\gamma_1) + \h(\gamma_2) \le 2C_2,
$$
and (also noticing \eqref{eq:m low1})
$$
B=m.
$$
Directly from \eqref{eq:C1}, we know that
$$
\prod_{j=1}^{t_1-1}\h(\eta_{1j})\cdot
\prod_{j=1}^{t_2-1}\h(\eta_{2j}) \le C_1^2.
$$
Using Lemma~\ref{lem:fact}, \eqref{h_LR_L-bound}, \eqref{eq:C1} and \eqref{eq:W low1} and noticing $e_1e_2 \le n^2 / 4$, we obtain
\begin{align*}
\h(\xi_1/\xi_2) & \le \h(\xi_1) + \h(\xi_2) = \frac{1}{m} \h(\xi_1^m) + \frac{1}{m} \h(\xi_2^m) \\
&\le \frac{1}{2m}(n^2 C_0^2 W + 4C_2) + \frac{1}{2}(t_1 + t_2 - 2)C_1 \\
& \le \left(\frac{1}{2}n^2C_0^2 + \frac{2C_2}{W} + \frac{mnsC_1}{W} \right) \frac{W}{m} \\
& \le \left(\frac{1}{2}n^2C_0^2 + \frac{2C_2}{C_3} + \frac{2dn^2sC_1}{\log 2} \right) \frac{W}{m} \\
& \le C_4 \cdot \frac{W}{m},
\end{align*}
where we also use $n \ge 2, t_1 \le ns, t_2 \le ns, d \le 2s$, and
$$
C_4 = (4^{n} n^{4} s)^{4ns} e^{4nd \h(f) } |D_K|^{2n} (\log^* P_S)^{ns-1}.
$$
Hence, we have
\begin{equation*}  \label{eq:Theta}
\Theta \le 2C_2 \cdot C_4 \cdot \frac{W}{m}  \cdot C_1^2 \le C_5 \cdot \frac{W}{m},
\end{equation*}
where
$$
C_5 = 2 \cdot 1200^{2} (4^{6n}n^{25}s^8)^{2ns} e^{12nd \h(f) } |D_K|^{6n} (\log^* P_S)^{3ns}\log^*N_S(b).
$$
We clearly have
$$
\frac{N(v_0)}{\log N(v_0)} \le \frac{N(v_0)}{\log 2} \le \frac{2P_S^{[L:K]}}{\log 2}  \le 3P_S^{n(n-1)} \le 3P_S^{n^2},
$$
where the factor $2$ comes from the case when there is no finite place in $T$.
In addition, using $d_L \le n(n-1)d \le 2n(n-1)s$  we get
$$
c  \le C_6 := (32en^2 s)^{6ns+3}.
$$

Therefore, inserting the above relevant estimates into \eqref{eq:lower} we obtain
\begin{equation*}
\left| 1- \left(\frac{x-\al_1}{x-\al_2}\right)^{e_1e_2h_{L_1}h_{L_2}} \right|_{v_0}
\ge   \exp\Big(- 3C_5 C_6 P_S^{n^2}  \frac{W}{m} \log m  \Big).
\end{equation*}
A comparison with the upper bound \eqref{eq:upper}, together with \eqref{eq:m low1}
and $t \le n^2 s$, gives
\begin{align*}
\frac{m}{\log m} & \le 6tC_5 C_6 P_S^{n^2}  \\
& \le 6 n^2 s  C_5 C_6 P_S^{n^2} \le C,
\end{align*}
where
$$
C= 4^{12n^2 s}(10n^2 s)^{38ns} e^{12nd\h(f)} |D_K|^{6n} P_S^{n^2} (\log^* P_S)^{3ns}\log^*N_S(b).
$$
Note that the real function $z/\log(z)$ is strictly increasing when $z \ge e$, and
$$
\frac{2C\log C}{ \log (2C\log C)} = \frac{2C\log C}{ \log C + \log (2\log C)} \ge \frac{2C\log C}{ \log C + \log C} = C.
$$
Hence, we obtain
$$
m \le 2C\log C,
$$
as claimed in the theorem.

\begin{remark}  \label{rem:improve}
Let us briefly discuss the case where $S$ is composed solely of infinite places.
Then, $v_0$ is infinite and, using the calculations above and $B=m$,
we derive from \cite[Theorem 2.1]{BuEMS} the lower bound
$$
\left| 1- \left(\frac{x-\al_1}{x-\al_2}\right)^{e_1e_2h_{L_1}h_{L_2}} \right|_{v_0}
\geq \exp\bigl(- C_7 C_2 \frac{W}{m} \log \frac{m}{2 C_2} \, \bigr),
$$
where
$$
C_7 = c(n,d) e^{8nd \h(f) } |D_K|^{4n}
$$
for some effectively computable $c(n,d)$ depending only on $n$ and $d$. It then follows from \eqref{eq:upper} that
$$
\frac{W d_L}{2 t} \le C_7 C_2 \frac{W}{m} \log \frac{m}{2 C_2},
$$
thus,
$$
\frac{m}{2 C_2} \le  t C_7  \log \frac{m}{2 C_2}.
$$
Hence, we can get an upper bound for $m$ which is linear in $C_2$, thus, linear in $\log^*N_S(b)$.
\end{remark}

\section{Comments}  
\label{sec:application}

The effective or explicit upper bounds about superelliptic equations have many applications in number theory; 
see, for instance, the citations of \cite{BEG, B6, Bug2, Voutier2}. 
Here, we want to mention their applications in multiplicative dependence of algebraic numbers. 

Recall that given  non-zero complex numbers $\alpha_1, \ldots, \alpha_n \in \C^*$, we say that they are \textit{multiplicatively dependent} if there exist integers $k_1,\ldots,k_n \in \Z$, not all zero, such that
\begin{equation*}
\alpha_1^{k_1}\cdots \alpha_n^{k_n} = 1;
\end{equation*}
and we say that they are \textit{multiplicatively dependent modulo $G$}, where $G$ is a subset of $\C^*$,
if there exist integers $k_1,\ldots,k_n \in \Z$, not all zero, such that
\begin{equation*}
\alpha_1^{k_1}\cdots \alpha_n^{k_n} \in G.
\end{equation*}

In a forthcoming paper of ours \cite[Theorem 1.5]{BA24}, using Theorem~\ref{thm:GenS-T} (especially, when $b \in O_S^*$, the bound does not depend on $b$), we have established a result on
multiplicative dependence of rational values modulo approximate finitely generated groups. 
A weak version of this result asserts that: given a finitely generated subgroup $\Gamma$ of $K^*$ and given two polynomials $f_1, f_2 \in K[X]$ both having at least two distinct roots, if they cannot multiplicatively generate a power of a linear fractional function, then there are only finitely many elements $\alpha \in K$ 
such that $f_1(\alpha)$ and $f_2(\alpha)$ are multiplicatively dependent modulo $\Gamma$.

In addition,  the authors in \cite[Theorem 1.7]{BOSS} have used essentially the versions of Theorems~\ref{T_LeVeque} and \ref{thm:GenS-T} about simple polynomials in \cite{BEG} to show that:  given a finitely generated subgroup $\Gamma$ of $K^*$ and given a polynomial $f \in K[X]$ only having simple roots, under some minor condition, there are only finitely many elements $\alpha \in K$ such that  for some distinct integers $m,n \ge 1$, the values $f^{(m)}(\alpha)$ and $f^{(n)}(\alpha)$ are multiplicatively dependent modulo $\Gamma$. Here, $f^{(m)}$ denotes the $m$-th iterate of $f$. 

Now, using our Theorems~\ref{T_LeVeque} and \ref{thm:GenS-T}, one can extend the above result to the case when $f$ has multiple roots.

\section*{Acknowledgement}

\noindent Jorge Mello and Alina Ostafe were partially supported by the
Australian Research Council Grants DP180100201 and DP200100355.
Min Sha was supported by the Guangdong Basic and Applied Basic Research Foundation (No. 2022A1515012032).
The research of Attila B\'erczes and K\'alman Gy\H ory was supported in part by grants K128088 and ANN130909 of the Hungarian National Foundation for Scientific Research.
Alina Ostafe also gratefully acknowledges the generosity and hospitality of the Max Planck Institute for Mathematics where parts of her work on this project were developed.

\end{document}